\documentclass[reqno]{amsart}
\usepackage{amsmath,amssymb}
\usepackage{latexsym}
\usepackage{dsfont}
\usepackage[mathscr]{euscript}
\usepackage{graphicx}
\usepackage{enumitem}
\usepackage[english]{babel}
\usepackage{hyperref}
\usepackage[title]{appendix}
\usepackage{xcolor}

\usepackage[small]{caption}

\usepackage{tikz}

\makeatletter
\@addtoreset{equation}{section}
\makeatother

\newtheorem{theorem}{Theorem}[section]
\newtheorem{lemma}[theorem]{Lemma}
\newtheorem{proposition}[theorem]{Proposition}
\newtheorem{corollary}[theorem]{Corollary}
\newtheorem{definition}[theorem]{Definition}
\newtheorem{remark}[theorem]{Remark}

\numberwithin{equation}{section}

\newcommand{\mc}[1]{{\mathcal #1}}

\newcommand{\bb}[1]{{\mathbb #1}}

\newcounter{as}[section]

\renewcommand{\>}{\rangle}

\begin{document}

\title{Additive functionals of exclusion processes from non-equilibrium }

\author{Luiz Renato Fontes}

\address{\noindent IME-USP, Rua do Mat\~ao 1010, CEP 05508-090, S\~ao Paulo, Brazil.
  \newline e-mail: \rm \texttt{lrfontes@usp.br} }

\author{Tiecheng Xu}

\address{\noindent IME-UFBA, Av.Milton Santos, s/n, Campus Universitario de Ondina, CEP 40170-110, Salvador, Brazil.
  \newline e-mail: \rm \texttt{xutcmath@gmail.com} }

\begin{abstract}
Consider the weakly asymmetric simple exclusion processes on the one-dimensional torus. We study the non-equilibrium fluctuation of a class of additive functionals, and show that its scaling limit  is a Gaussian process. The proof is mainly based on the results obtained and techniques developed by Jara and Menezes [Non-equiliburim fluctuations of interacting particle systems, Preprint].
\end{abstract}
\keywords{Non-equilibrium fluctuations, occupation time, exclusion processes}

\maketitle 
\section{Introduction}
 Given a Markov process $\{\eta_t:t\geq 0\}$ with state space $\Omega$, and a proper function $f:\Omega\to\bb R$, one would like to understand the long-time behavior of 
$$\Sigma(t)\,:=\,\int_0^tf(\eta_s)ds$$ 
under a suitable space-time rescaling. Since the seminal work of Kipnis and Varadhan \cite{kv}, plenty of results has been obtained on the additive functionals of Markov processes, especially of a particular class of Markov processes, the interacting particle systems. For the interacting particle systems on the one-dimensional lattice, we refer to \cite{s00}\cite{s03}\cite{b04}\cite{fgn14} on the scaling limits of the additive functionals of exclusion processes of various cases, and refer to \cite{qjs02} for those of zero range processes. We also point out a landmark work \cite{gj13} of Gon\c calves and Jara, where the authors obtained the Central Limit Theorem of a quite general class of additive functionals of one-dimensional, conservative, stationary interacting particle systems. 

To the best of our knowledge, all the previous results in this topic are obtained under the assumption that the interacting particle system starts from its invariant measure. The purpose of this work is to present the first result on the Central Limit Theorem of the additive functionals of non-equilibrium systems.  

The interacting particle systems that we work on are a sequence of spatially inhomogeneous, weakly asymmetric simple exclusion processes on the one-dimensional discrete torus $\bb T_n$ with $n$ sites. The dynamics of these processes can be informally described as follows. Fix a smooth function $F$ defined on the one-dimensional continuous torus. The particle at site $x\in\bb T_n$ attempts to jump to the right neighbour site $x+1$  with rate $1+\frac{1}{n}F(\frac{x}{n})$, and attempts to jump to the left neighbour site $x-1$ with  rate $1-\frac{1}{n}F(\frac{x-1}{n})$. The system obeys the exclusion rule, namely, at most one particle is allowed at each site. When the destination site of the attempted jump is occupied, the jump is suppressed.

Let us denote this weakly asymmetric simple exclusion processes on $\bb T_n$ by $\{\eta_t: t\geq 0\}$. The system is taken to start from a  non-equilibrium state, which consists of a Bernoulli product measure with a smooth profile. Fix a local function $h$ with finite support and a smooth function $q_{\cdot}:[0,T]\to\bb R_+$ that is bounded away from zero.  Let $\varphi_h(\beta)$ be the expectation of $h$ with respect to Bernoulli product measure with constant density $\beta$. In this work we prove that the limit as $n\to\infty$ of the following process 
\begin{equation}\label{1time}
\Gamma^h_n(t)\,:=\,\sqrt{n}\int_0^{t}\frac{h(\eta_{sn^2})-\varphi_h(\rho_s(0))}{q_s}ds,\,\, 0\leq t\leq T
\end{equation}
is a Gaussian process,
where $\rho_s$ is the solution of the corresponding hydrodynamic limit equation obtained under the diffusive time scale: 
$$\partial_s\rho=\Delta \rho\,-\,2\nabla\big\{\rho(1-\rho)F\big\}.$$
The dependence of the limit process on the local function $h$ is also characterized. 

The main idea of our proof is to follow the approach adopted in \cite{gj13}, which can be roughly described as follows. Once the limit of the denstiy fluctuation field is derived, one then needs to choose a proper test function for the density field such that $\Gamma_n^h$ can be well approximated by the density field acting on this test function. In this way we can capture the feature of $\Gamma_n^h$ through the existing thorough understanding of the density fluctuation field. 

With regards to the non-equilibrium fluctuations of the density field, in general it is very difficult to derive and has been one of the main open problems in statistical mechanics. It was just recently that Jara and Menezes \cite{jm} made a major breakthrough by introducing a robust method, which is a refinement of relative entropy method of Yau\cite{y91},  to derive the non-equilibrium fluctuations of the interacting particle systems. To avoid repeating their arguments to obtain the limit of  non-equilibrium fluctuations of the density field,  in the present work we choose to work on the same model studied in \cite{jm}.  An intesting feature of this model is that, as quoted in \cite{jm},`` this model does not have known explicit invariant measures, and therefore not even in the equilibrium case this model is tractable by previous methods. "

The approximation of $\Gamma_n^h$ is realized through a local replacement lemma(see Lemma \ref{ZG}). A typical local replacement lemma tells that  a microscopic centered variable in the form of
 $$\overline{\eta}_{sn^2}(x):=\eta_{sn^2}(x)-\rho_s(x/n)$$
 can be replaced by a macroscopic variable, with a negligible cost after taking time integral. Nevertheless there are several obstacles to overcome for a suitble local replacement lemma meeting our needs. 

The first one is the choice of the macroscopic variable in the non-equilibrium setting. Let us  first consider the simple case that $h(\eta_{sn^2})$ is chosen to be the centered variable $\overline{\eta}_{sn^2}(0)$. In the equilibrium particle systems, an usual choice of the macroscopic variable is the  average of $\overline{\eta}_{sn^2}(x)$ in a box $(0,\varepsilon n]\cap\bb Z$:
$$\frac{1}{\varepsilon n}\sum_{x=1}^{\varepsilon n}\overline{\eta}_{sn^2}(x).$$
However it is not clear how to obtain a good enough bound on the cost of doing such a replacement in the non-equilibrium setting. We find out that a more practicable macroscopic variable to replace $\overline{\eta}_{sn^2}(0)$ is
$$\frac{\rho_s(0)[1-\rho_s(0)]}{\varepsilon n}\sum_{x=1}^{\varepsilon n}\frac{\overline{\eta}_{sn^2}(x)}{\rho_s(x/n)[1-\rho_s(x/n)]}.$$
Consider now a general local function $h$. Gon\c calves and Jara\cite{gj13} use the equivalence of ensembles and spectral gap tools to find out the appropriate macroscopic variable to perform replacement, if the particle system starts from its stationary product measure.  However it is not clear how to apply their strategy for particle systems without stationary product measures or starting from non-stationary states. Our idea is to decompose the local function $h$ into the linear sum of basis of different degrees. A basis of degree one is of the form $\overline{\eta}(x)$, for which we have just discussed how to perform the replacement. Basis of degree at least two are of the forms $\overline{\eta}(0)\overline{\eta}(1), \overline{\eta}(0)\overline{\eta}(1)\overline{\eta}(2)$, etc. These basis of high orders will be shown to contribute nothing in the replacement, under the space-time scaling as given in \eqref{1time}. In this way we demonstrate the macroscopic variable to replace $h(\eta_{sn^2})$ totally comes from contributions of the degree one basis in the decomposition.  

The second one is under which norm to estimate the cost. In the equilibrium systems, $L^2$ norm is an obvious choice because of the famous Kipnis-Varadhan inequality. In the non-equilibrium systems, a Kipnis-Varadhan type inequality was recently discovered in \cite{efx}, which works well for the degree one basis, i.e. the case $h(\eta)=\overline{\eta}(x)$. A generalization of the non-equilibrium Kipnis-Varadhan type inequality for a general local function $h$ seems to be out of reach for the moment. Inspired by the techniques developed in \cite{jm}\cite{jm1}, instead of estimating $L^2$ norm, we provide a convenient mechanism to obtain a $L^\lambda$ estimate on the cost of doing replacement for a general local function $h$, for any $\lambda\in(1,2)$. 

The paper is organized in the following way. In Section \ref{sec2} we define our model, recall the hydrodynamic limit and fluctuations of the density obtained in \cite{jm}, and state our main results about scaling limits of additive functionals. In Section \ref{sec3} we  prove the key result of our paper, the local replacement lemma out of equilibrium. Using the local replacement lemma, we prove in Section \ref{sec4} the tightness of three sequences of processes: the first one is used to characterize the limiting process, the second one is an error process which will be shown to converge to a zero process in the limit,  the third one is the sequence $\{\Gamma^h_n(t): t\in[0,T]\}_{n\in\bb N}$.  In Section \ref{sec5} we prove the theorems stated in Section \ref{sec2}. In the Appendix we present
some properties of subgaussian random variables.
\section{The Model and Main results}\label{sec2}
\subsection{WASEP}

For each $n\in\bb N$, denote by $\bb T_n=\bb Z/n\bb Z$ the one-dimensional discrete torus. Consider the exclusion process on $\bb T_n$, denoted by $\{\eta_t: t\geq 0\}$, whose generator $L_n$ acting on a function $f:\{0,1\}^{\bb T_n}\to\bb R$ is given by
$$L_n f(\eta)\,=\,\sum_{x,y\in\bb T_n}r_n(x,y)\eta(x)(1-\eta(y)) [f(\eta^{x,y})\,-\,f(\eta)]$$
In this formula and below, the configurations are represented by the Greek letters $\eta$, $\xi$, so that $\eta(x)=1$ if site $x\in \bb T_n$ is occupied for the configuration $\eta$ and $\eta(x)=0$ otherwise. The symbol $\eta^{x,y}$ represents the configuration obtained from $\eta$ by switching the occupation of sites $x$ and $y$:
\begin{equation*}
(\eta^{x,y})(z)=
\begin{cases}
\eta(y) & \mbox{ if } z=x\\
\eta(x) & \mbox{ if } z=y\\
\eta(z) & \mbox{ if } z\neq x,y\;. 
\end{cases}
\end{equation*} 

Denote the state space $\{0,1\}^{\bb T_n}$ by $\Omega_n$. Let $\bb T=\bb R/\bb Z$ and fix a smooth function $F:\bb T\to\bb R$. We choose the jump rate function $r_n:\bb T_n\times \bb T_n\to\bb R_{\geq 0}$ such that $r_n(x,x+1)=1+ \frac{1}{n}F(\frac{x}{n})$, $r_n(x+1,x)=1- \frac{1}{n}F(\frac{x}{n})$ and $r_n(x,y)=0$ if $\lvert x-y\rvert>1$. In order to guarantee that $r_n$ is non-negative, since $F$ is a bounded function, we assume that $n$ is sufficiently large. This process is called the \textit{weakly asymmetric exclusion process}(WASEP).  In particular if $F\equiv 0$, the process $\{\eta_t^n: t\geq 0\}$ becomes the \textit{simple symmetric exclusion process}(SSEP). 

\subsection{Scaling limits}

Consider an integrable function with respect to the Lebesgue measure $\rho(\cdot):\bb T\to\bb R_+$ . Denote by $\nu_{\rho(\cdot)}^n$ the \textit{Bernoulli product measure with slowly varying parameter} associated to the profile $\rho(\cdot)$:
\begin{equation}\label{defnu}
\nu_{\rho(\cdot)}^n\{\eta:\eta(x)=1\}\,=\,\rho(x/n), \,\,\,\,x\in  \bb T_n.
\end{equation} Throughout this article, we  shall fix a smooth function $\rho_0$ such that there exists $\varepsilon_0\in(0,1/2)$ such that $\rho_0(u)\in (\varepsilon_0,1-\varepsilon_0)$ for every $u\in\bb T$. This function $\rho_0$ will be used as the initial profile of our model.

Given two probability measures $\mu$ and $\pi$ on the same probability space $E$ such that: $\mu$ is absolutely continuous with respect to $\pi$, the relative entropy of $\mu$ with respect to $\pi$ is defined by 
$$H(\mu|\pi)\,=\,\int_E\frac{d\mu}{d\pi}\log\frac{d\mu}{d\pi}d\pi,$$
where $\frac{d\mu}{d\pi}$ represents the Radon-Nikodym derivative of $\mu$ with respect to $\pi$.

 Let $\mc M$ be the space of positive measures on $\bb T$ with total mass bounded by one, endowed with the weak topology. Let
$\pi^{n}_{t} \in \mc M$ be the empirical measure on $\bb T$ obtained by rescaling time by $n^2$, rescaling space by $n^{-1}$, and assigning mass $n^{-1}$ to each particle, i.e.,
\begin{equation}\label{f01}
\pi^{n}_{t}(\eta,du) \;=\; \frac{1}{n} \sum _{x\in \bb T_n} \eta_{tn^2} (x)\,
\delta_{x/n}(du)\,,
\end{equation}
where $\delta_u$ is the Dirac measure concentrated on $u$. Let $D(\bb R_+, \Omega_n)$ be the path space of c\`adl\`ag trajectories with values in $\Omega_n$. Given a measure $\mu_n$ on $\Omega_n$, denote by $\bb P_{\mu_n}$( resp. $\bb E_{\mu_n}$) the probability measure(resp. expectation) on $D(\bb R_+, \Omega_n)$ induced by the
initial state $\mu_n$ and the Markov process $\{\eta_t : t\ge 0\}$. 

Suppose that the process $\{\eta_t^n: t\geq 0 \}$ starts from an initial distribution $\mu^n$. We shall assume $\mu^n$ is close to the Bernoulli product measure $\nu^n_{\rho_0(\cdot)}$ in the sense of relative entropy. To simplify the notation, we write
$$H_n(0)\,=\,H(\mu^n\mid \nu^n_{\rho_0(\cdot)}).$$ In order to observe the evolution of the density in the macroscopic level,  the process has to be accelerated by $n^2$.  The next theorem on the hydrodynamic limit of WASEP is a well known result. See Proposition 2.1 of \cite{jm} for instance.
\begin{theorem}\label{hdl}
Assume that $H_n(0)=o(n)$.
Then for every $t\in [0,T]$, every continuous function $f:\bb T\to\bb R$ and every $\delta>0$,
$$\lim_{n\to\infty}\bb P_{\mu^n}\Big\{\eta\,:\,\Big|\frac{1}{n}\sum_{x\in\bb T_n}f\big(\frac{x}{n}\big)\eta_{tn^2}(x)\,-\,\int_{\bb T}f(u)\rho_t(u)du\Big|>\delta\Big\}\,=\,0,$$
where $\rho_t(u)$ is the unique solution of the heat equation 
\begin{equation}\label{pde}
\begin{cases}
\partial_t\rho_t(u)=\partial_{uu}^2 \rho_t(u)\,-\,2\partial_u\big\{\rho_t(u)(1-\rho_t(u))F(u)\big\}, \quad u\in \bb T \\
\rho(0,\cdot)=\rho_0(\cdot) \,
\end{cases}
\end{equation}

\end{theorem}

The theorem above can be interpreted as the law of large numbers of the empirical measure $\pi^n$. 
The next result is the associated central limit theorem, which is the non-equilibrium fluctuation of the density field.  Let us start by introducing the density fluctuation field. The density  fluctuation field $X^n$ is given by 
\begin{equation}\label{dff}
X_t^n(f)\,:=\,\frac{1}{\sqrt{n}}\sum_{x\in\bb T_n}\big[\eta_{tn^2}(x)\,-\,\rho_t\big(\frac{x}{n}\big)\big]f(\frac{x}{n})
\end{equation}
for every test function $f\in C^{\infty}(\bb T)$ and $t\geq 0$.

We now define the space where field $X_t^n$ lives. Given an integer $k\in \bb Z$, denote by $H_{k}(\bb T)$ the Sobolev space:
$$H_k(\bb T)\,:=\,\Big\{f\in L^2(\bb T): \,\sum_{n\in \bb Z}(1+n^2)^k\lvert \widehat{f}(n)\rvert^2<\infty\Big\}$$
where $\widehat{f}$ is the Fourier series of $f$. The norm $\|\cdot\|_k$ of $H_k(\bb T)$ is naturally given by
$$\|f\|_k\,=\, \sum_{n\in \bb Z}(1+n^2)^k\lvert \widehat{f}(n)\rvert^2.$$
It is shown in \cite{jm} that the process $\{X_t^n, t\geq 0\}$ takes value in $H_{-k}(\bb T)$ for any $k>1/2$. The next theorem gives the limit of the density fluctuation field $X^n$ as $n\to\infty$.
\begin{theorem}\label{CLT}[Theorem 2.4 and 7.1 in \cite{jm}]
Fix  $k>3$ and assume $H_n(0)=O(1)$. Assume furthermore that the random variable $X_0$ takes value in $H_{-k}(\bb T)$ and $X_0^n$ converges to $X_0$ in law with respect to the topology of $H_{-k}(\bb T)$. Then  $\{X_t^n, 0\leq t\leq T\}$ converges in law to the process $\{X_t, 0\leq t\leq T\}$, which is the solution of the stochastic heat equation
\begin{equation}\label{spde}
\partial_t X_t\,=\, \nabla\Big( \nabla X_t\,-\,2X_t(1-\rho_t)F\,+\,\sqrt{2\rho_t(1-\rho_t)}\dot{W}_t\Big)
\end{equation}
with initial condition $X_0$, where $\dot{W}_t$ is a space-time one-dimensional white noise.
\end{theorem}
In the last theorem, a stochastic process $\{X_t: t\in [0,T]\}$ is said to be the solution of \eqref{spde} if for any $f\in C^\infty\big([0,T]\times \bb T\big)$, the process $\{M_t(f): t\in[0,T]\}$ defined as 
\begin{equation}\label{mardef}
M_t(f)\,=\,X_t(f_t)\,-\,X_0(f_0)\,-\,\int_0^t\,X_s\big((\partial_s+\bb L_s)f_s\big)ds
\end{equation}
is a continuous martingale with respect to the filtration $\mc F_t=\sigma\{X_s(f_s): s\leq t, f\in C^\infty\big([0,T]\times \bb T\big)\}$,
whose quadratic variation is  
\begin{equation}\label{qvar}
\<M_t(f)\>=\int_0^t\int2\rho_s(u)(1-\rho_s(u))|\partial_u f_s(u)|^2duds.
\end{equation}
The generator $\bb L_t$ in \eqref{mardef} is defined as
\begin{equation}\label{Lt}
\bb L_t f(u)\,:=\,\Delta f(u)\,+\,2\big(1-2\rho_t(u)\big)F(u) f'(u)
\end{equation}
for any $f\in C^\infty(\bb T)$ and $u\in\bb T$.

\subsection{Occupation time at the origin}
To illustrate ideas, we first consider a particular case that $h(\eta_{sn^2})=\overline{\eta}_{sn^2}(0) $. With this choice of $h$, $\Gamma_n^h(t)$ becomes the so-called rescaled generalized occupation time at the origin, which we denote by $\Gamma_n(t)$:
$$ \Gamma_n(t)\,=\,\sqrt{n}\int_0^{t}\frac{\eta_{sn^2}(0)-\rho_s(0)}{q_s}ds,$$
where $q_s:[0,T]\to\bb R_+$ is a fixed smooth function such that $M^{-1}\leq q_s\leq M$ for some positive number $M$ and for all $s\in[0,T]$. 

We follow a classical idea to study the occupation time  in one dimensional particle systems, which is to take a test function that approximates well the Dirac delta function, in the density fluctuation field \eqref{dff}

The choice of the approximation test function to the Dirac function $\delta_0$ turns out a little bit tricky. It is helpful to consider the approximation in the discrete setting first. Let $\overline{\eta}_{tn^2}(x)$ be the centered occupation variable at site $x$ at time $tn^2$:
$$\overline{\eta}_{tn^2}(x)\,:=\,\eta_{tn^2}(x)\,-\,\rho_t\big(\frac{x}{n}\big).$$
In order to link the scaling limit of occupation time with the denstiy fluctuation field, we would like to replace $\overline{\eta}_{sn^2}(0)$ inside $\Gamma_n(t)$ by some variable which depends on coordinates of $\eta$ in a small box of order $n$, with a negligible cost after taking time integration. As mentioned in introduction, a popular option is to replace $\overline{\eta}(0)$ by $\frac{1}{\varepsilon n}\sum_{x=1}^{\varepsilon n}\overline{\eta}(x)$, then let $\varepsilon\to 0$. However it is not clear how to control the cost of doing such a replacement. On the other hand, an integration by parts formula obtained in \cite{jm} permits to estimate  the cost of replacing $w_{sn^2}(x)$ by  $w_{sn^2}(y)$, where  
$$w_{sn^2}(x)\,:=\,\frac{\eta_{tn^2}(x)\,-\,\rho_t\big(\frac{x}{n}\big)}{\rho_t\big(\frac{x}{n}\big)\big(1-\rho_t\big(\frac{x}{n}\big)\big)} \quad \text{for all }\,\, x\in\bb T_n,$$
and $\rho_t$ is the hydrodynamic limit equation. 
This indicates that it may be a good idea to replace $w(0)$ by $\frac{1}{\varepsilon n}\sum_{x=1}^{\varepsilon n}w(x)$.
Therefore a more appropiate way to deal with the limit of $\Gamma_n(t)$ is as follows: we first write $\Gamma_n(t)$ in the form of 
$$\Gamma_n(t)\,=\,\sqrt{n}\int_0^{t}\frac{\rho_s(0)\big(1-\rho_s(0)\big)}{q_s}w_{sn^2}(0)ds,$$
then approximate it by 
\begin{equation}\label{disapr}
\begin{split}
&\sqrt{n}\int_0^{t}\frac{\rho_s(0)\big(1-\rho_s(0)\big)}{q_s}\frac{1}{\varepsilon n}\sum_{x=1}^{\varepsilon n}w_{sn^2}(x)ds\\
=\,&\sqrt{n}\int_0^{t}\sum_{x=1}^{\varepsilon n}\frac{\rho_s(0)\big(1-\rho_s(0)\big)}{q_s\,\rho_s\big(\frac{x}{n}\big)\big(1-\rho_s\big(\frac{x}{n}\big)\big)}\frac{1}{\varepsilon n}\overline{\eta}_{sn^2}(x)ds\quad \text{as}\,\, \varepsilon\to 0
\end{split}
\end{equation}
and estimate the cost. 

There is one more issue needs to be taken care: the test function for the density fluctuation field has to be smooth. Let $\mathds{1}_{A}$ be the indicator function of the set $A\subset \bb T$. In the above discussion, $\sqrt{n}w_{sn^2}(0)$ is actually replaced by 
$$X_s^n\Big(\frac{\varepsilon^{-1}\mathds{1}_{(0,\varepsilon)}(\cdot)}{\rho_s(\cdot)(1-\rho_s(\cdot))}\Big).$$
The test function inside the big brace is not smooth. Since $\rho_s$ is a smooth function, we only need to find a smooth version of function $\varepsilon^{-1}\mathds{1}_{(0,\varepsilon)}$. 

Based on the  intuition above, we are ready to introduce the approximation test function.  Let $\phi:\bb R\to\bb R$ be a nonnegative smooth function with support $(0,1)$ and has integral $1$. Define the mollifiers $\phi_\varepsilon(u):=\varepsilon^{-1}\phi(u/\varepsilon)$ for each $\varepsilon\in(0,1)$.  This smooth function $\phi_\varepsilon(u)$ will play the  role of $\varepsilon^{-1}\mathds{1}_{(0,\varepsilon)}$ when performing the replacement,  as we discussed previously.

Fix $T>0$. The expression in \eqref{disapr} inspires us to define the real-valued process $\{Z_{t,n}^\varepsilon:t\in[0,T]\}$ by 
\begin{equation}\label{Ztn}
Z_{t,n}^\varepsilon\,:=\,\int_0^tX^n_s\Big(\frac{\phi_\varepsilon(\cdot)\mc X(\rho_s(0)) }{q_s\mc X(\rho_s(\cdot))}\Big)ds.
\end{equation}
where $\mc X(\rho):= \rho(1-\rho)$ is the \textit{compressibility} of the system. As $n\to\infty$, by Theorem \ref{CLT}, $Z_{t,n}^\varepsilon$ converges in distribution to 
\begin{equation}\label{Zt}
Z_t^\varepsilon\,:=\,\int_0^tX_s\Big(\frac{\phi_\varepsilon(\cdot)\mc X(\rho_s(0)) }{q_s\mc X(\rho_s(\cdot))}\Big)ds.
\end{equation}
 
Our first result is the characterization of the limit of $Z_t^\varepsilon$ as $\varepsilon\to 0$.
\begin{theorem}\label{limit}
Assume that $X_0$ is a Gaussian random field and $H_n(0)=O(1)$. Then as $\varepsilon\to 0$, the process $\{Z_t^\varepsilon:t\in[0,T]\}$ converges in distribution  to a Gaussian process $\{Z_t:t\in[0,T]\}$, in space $C([0,T],\bb R)$ endowed with the uniform topology. 
\end{theorem}

The limiting process $Z_t$ above is also the scaling limit of the generalized occupation time $\Gamma_n(t)$. This is the content of our second result.
\begin{theorem}\label{time}
Assume that $X_0$ is a Gaussian random field and $H_n(0)=O(1)$. Then as $n\to\infty$, the process $\{\Gamma_n(t): t\in [0,T]\}$ converges in distribution, to the same Gaussian process $\{Z_t:t\in[0,T]\}$ as in Theorem \ref{limit}, in space $D([0,T], \bb R)$ endowed with uniform topology.
\end{theorem}

\subsection{Additive functionals}
Theorem \ref{time} can be extended from occupation time to a wide class of additive functionals. 

Fix a positive integer $R$. A  function $h:\Omega_n\to\bb R$ is said to be \textit{local} with support on $[-R,R]\cap \bb Z$, if $h(\sigma^z\eta)\,=\,h(\eta)$ for every integer $z\notin [-R,R]\cap\bb Z$, where $\sigma^z(\eta)$ is the configuration 
\begin{equation*}
(\sigma^z\eta)(x)=
\begin{cases}
\eta(x) & \mbox{ if } x\neq z\\
1-\eta(z) & \mbox{ if } x=z\, .
\end{cases}
\end{equation*} 
Given a local function $h$, let $\varphi_h(\beta)$ be the expectation of $h$ with respect to Bernoulli product measure $\nu_\beta^n$:
$$\varphi_h(\beta)\,:=\,E_{\nu_\beta^n}[h], \quad \beta\in[0,1]\,.$$
We would like to investigate the limit of the following additive functionals:
$$\Gamma_n^h(t)\,:=\, \sqrt{n}\int_0^{t}\frac{h(\eta_{sn^2})-\varphi_h(\rho_s(0))}{q_s}ds.$$

Similar to what  we did to deal with occupation time, the idea here is still to find a proper test function for the density fluctuation field, which allows us to perform a replacement for $\Gamma_n^h(t)$ and takes advantage of the convergence of the proecess $\{X^n_t, 0\leq t\leq T\}.$

In order to explain how to discover the appropriate  test function,  a careful analysis of the structure of the local function $h$ with a fixed finite support is needed. Given a local function $h$ with support on $[-R,R]\cap\bb Z$,  in Section \ref{sec3} we will show that $h(\eta_{sn^2})$ can be written as 
\begin{equation}\label{hdec}
\sum_{A\subset [-R,R]\cap\bb Z}c_s(A) \prod_{x\in A}\Big\{ \eta_{sn^2}(x)\,-\,\rho_s\big(\frac{x}{n}\big)\Big\},
\end{equation}
with convention 
$$\prod_{x\in \varnothing}\Big\{ \eta_{sn^2}(x)\,-\,\rho_s\big(\frac{x}{n}\big)\Big\} \,=\,1,$$
where the coefficient $c_s(A)$ satisfies
$$\sup_{0\leq s\leq T}\big\lvert c_s(\varnothing)-\varphi_h(\rho_s(0))\big\rvert\,\lesssim \,n^{-1},$$
$$\sup_{0\leq s\leq T}\big\lvert \sum_{|x|\leq R} c_s(\{x\})\,-\,\varphi'_h(\rho_s(0))\big\rvert\lesssim n^{-1}.$$
In view of what we did for occupation time previously, each $\sqrt{n}\big(\eta_{sn^2}(x)\,-\,\rho_s(x/n)\big)$ can be replaced by 
$$X_s^n\Big(\frac{\phi_\varepsilon(\cdot)\mc X(\rho_s(0))}{\mc X(\rho_s(\cdot))}\Big).$$
In addition, it will be shown that the terms $\prod_{x\in A}\Big\{ \eta_{sn^2}(x)\,-\,\rho_s(x/n)\Big\}$ with $|A|\geq 2$ contribute nothing in the replacement:
$$\limsup_{n\to 0}\sqrt{n}\int_0^t \sum_{\substack{A\subset [-R,R]\cap\bb Z\\|A|\geq 2}}c_s(A) \prod_{x\in A}\Big\{ \eta_{sn^2}(x)\,-\,\rho_s\big(\frac{x}{n}\big)\Big\}ds\,=\,0.$$

Discussions above lead us to define the following real-valued processes 
\begin{equation}\label{Ztnh}
Z_{t,n}^{h,\varepsilon}\,:=\,\int_0^tX^n_s\Big(\frac{\phi_\varepsilon(\cdot)\mc X(\rho_s(0)) \varphi'_h(\rho_s(0))}{q_s\mc X(\rho_s(\cdot))}\Big)ds,
\end{equation}
\begin{equation}\label{Zth}
Z_{t}^{h,\varepsilon}\,:=\,\int_0^tX_s\Big(\frac{\phi_\varepsilon(\cdot)\mc X(\rho_s(0)) \varphi'_h(\rho_s(0))}{q_s\mc X(\rho_s(\cdot))}\Big)ds.
\end{equation}
From definition we see that $Z_t^\varepsilon$(resp. $Z_{t,n}^\varepsilon$) is just $Z_{t}^{h,\varepsilon}$(resp. $Z_{t,n}^{h,\varepsilon}$) choosing $h=\eta_0$ in particular.

Main results of our paper are the following two theorems. Apparently these two theorem are generalizations of Theorem \ref{limit} and Theorem \ref{time} respectively.
\begin{theorem}\label{limith}
Assume that $X_0$ is a Gaussian random field and $H_n(0)=O(1)$. Then as $\varepsilon\to 0$, the process $\{Z_{t}^{h,\varepsilon}:t\in[0,T]\}$ converges in distribution  to a Gaussian process $\{Z_{t}^h:t\in[0,T]\}$, in space $C([0,T],\bb R)$ endowed with the uniform topology.
\end{theorem}

\begin{theorem}\label{timeh}
Assume that $X_0$ is a Gaussian random field and $H_n(0)=O(1)$. Then as $n\to\infty$, the process $\{\Gamma_{n}^h(t): t\in [0,T]\}$ converges in distribution, to the same Gaussian process $\{Z^h_t:t\in[0,T]\}$ as in Theorem \ref{limith}, in space $D([0,T], \bb R)$ endowed with uniform topology.
\end{theorem}

\section{A local replacement lemma}\label{sec3}
The main purpose of this section is to prove a local replacement lemma, which is Lemma \ref{ZG} in subsection \ref{sec33}. In subsection \ref{sec31} we decompose the local function $h$ into the sum of basis of different degrees. This decomposition indicates why $Z_{t,n}^{h,\varepsilon}$ is the appropriate process to be chosen in the replacement lemma. In subsection \ref{sec32} we recall some results from \cite{jm}, which will be used to prove the local replacement lemma. In subsection \ref{sec34} and \ref{sec35}, we provide proofs of some estimates needed in the proof of Lemma \ref{ZG}.
\subsection{Decomposition of the local function $h$}\label{sec31}
For a subset $A\subset \bb Z$, define 
$$\overline{\eta}_{sn^2}(A)\,:=\, \prod_{x\in A}\Big\{ \eta_{sn^2}(x)\,-\,\rho_s\big(\frac{x}{n}\big)\Big\},$$
with convention $\overline{\eta}_{sn^2}(\varnothing)=1.$
It is not hard to see that the local function $h(\eta_{sn^2})$ can be represented in the form of linear combination of $\overline{\eta}_{sn^2}(A)$:
$$h(\eta_{sn^2})\,=\,\sum_{A\subset \bb T_n}c_s(A) \overline{\eta}_{sn^2}(A).$$
Since we assume that the local function $h$ is supported on $[-R,R]\cap\bb Z$ for some fixed  positive integer $R$, the cofficient $c_s(A)$ is equal to zero if $A$ contains some element $x$ such that $x\notin [-R,R]\cap\bb Z$. In other words, we can write 
\begin{equation}\label{dech}
h(\eta_{sn^2})\,=\,\sum_{A\subset [-R,R]\cap\bb Z}c_s(A) \overline{\eta}_{sn^2}(A).
\end{equation}

We claim that there exists a constant $C$ independent of $s$ such that 
\begin{equation}\label{claimdec}
\begin{split}
&\big\lvert\varphi'_h(\rho_s(0))\,-\, \sum_{|x|\leq R} c_s(\{x\})\big\rvert\,\leq\, Cn^{-1},\\
&\big\lvert\varphi_h(\rho_s(0))\,-\, c_s(\varnothing)\big\rvert\,\leq\, Cn^{-1}.
\end{split}
\end{equation}
Indeed, since $\rho_s$ is smooth and $R$ is fixed, replacing every $\rho_s(x/n)$ from $\overline{\eta}_{sn^2}(A)$ by $\rho_s(0)$ in \eqref{dech}, we have 
$$h(\eta_{sn^2})\,=\,\sum_{A\subset [-R,R]\cap\bb Z}\widetilde{c}_s(A) \prod_{x\in A}\Big\{ \eta_{sn^2}(x)\,-\,\rho_s(0)\Big\},$$
with the new cofficient $\widetilde{c}_s(A)$ satisfying 
$$\sup_{0\leq s\leq T}\sup_{A\subset [-R,R]\cap\bb Z}|c_s(A)-\widetilde{c}_s(A)|\,\lesssim\,n^{-1}.$$
To prove the claim, it remains to observe that 
$$\varphi_h(\rho_s(0))\,=\,\widetilde{c}_s(\varnothing), \quad   \varphi'_h(\rho_s(0))\,=\, \sum_{|x|\leq R} \widetilde{c}_s(\{x\}).$$

\subsection{Entropy estimate and related results}\label{subsecEnt}\label{sec32}

 For any function $f:\bb T_n\to\bb R$, define the Dirichlet form with respect to a probability measure $\mu$ on $\Omega_n$ by
$$D(f;\mu)\,=\,\sum_{\eta\in\Omega_n}\sum_{x,y\in\bb T_n}[f(\eta^{x,y})\,-\,f(\eta)]^2 \mu(\eta).$$
Note that here $\frac{1}{2}D(f;\mu)$ is not equal to $-\<f, L_nf \>_{\mu}$. However they satisfy the following relation: for every probability measure $\mu$, 
\begin{equation}\label{relD}
-\<f,L_nf\>_{\mu}\,\geq\, \frac{1}{2}D(f;\mu).
\end{equation}

Recall the definition \eqref{defnu}. For every $t\in(0,T]$, let $\mu_t^n$ be the Bernoulli product measure with slowing varying parameter $\rho_t(\cdot)$ and $\nu^n_{1/2}$ be the Bernoulli product measure with slowing varying parameter $1/2$. $\mu_t^n$ is going to be the time-dependent reference measure to approximate the distribution of $\eta_{tn^2}$. 
Let $f_t^n=\frac{d\eta^n_{tn^2}}{d\mu_t^n}$ and let  $H_n(t)$ be the relative entropy of the law of $\eta_{tn^2}$ with respect to the reference measure $\mu_t^n$: 
$$H_n(t)\,=\,\int_{\Omega_n} f_t^n\log f_t^n d\mu_t^n.$$

The following theorem gives a sharp upper bound on the relative entropy $H_n(t)$, which depends on $H_n(0)$ and $\rho_0(\cdot)$. Recall that we assume that $\rho_0\in (\varepsilon_0,1-\varepsilon_0)$ for some small $\varepsilon_0>0$. It is shown in Lemma B.3 in \cite{jm} that $\rho(t,\cdot)\in (\varepsilon_1,1-\varepsilon_1)$ for some $\varepsilon_1\in (0,1/2)$. For the simplicity of notation, we just assume that $\rho_t(u)\in (\varepsilon_0,1-\varepsilon_0)$ for all $t\in[0,T]$ and all $u\in\bb T$. It is also proved in Proposition B.1 in \cite{jm} that $\rho_t$ is smooth. Therefore we can assume that there exists $\kappa>0$ such that 
$$\Big\lvert \rho_t\big(\frac{x+1}{n}\big)\,-\,\rho_t\big(\frac{x}{n}\big)\Big\rvert \leq \frac{\kappa}{n}, \quad \text{for all \,\,} x\in T_n, \,\, \text{all}\,\,n\in\bb N,  \,\, \text{and all\,\,} t\in\bb [0,T].$$

\begin{theorem}[Theorem 2.2 in \cite{jm}]\label{entropy}
For every $t\in [0,T]$, there exists a constant $C=C(\varepsilon_0, T,\kappa)$\footnote{Throughout this article, the constant $C$ always depends on fixed parameters only, and may change from line to line.} such that
\begin{equation*}
H_n(t)\,\leq\, C(H_n(0)\,+\,8).
\end{equation*}
\end{theorem}

The next lemma is one of the main steps in the  proof the above theorem.  It is a particular case  of Lemma 3.1 of \cite{jm} with $A=\{0\}$.
\begin{lemma}\label{main}
 Let $\nu^n_{\rho(\cdot)}$ be the Bernoulli product measure with slowly  varying parameter associated to a smooth profile $\rho(\cdot)$ which takes value in $(\varepsilon_0,1-\varepsilon_0)$.
Given a function $G:\bb T_n\to\bb R$, there exists a finite constant $C=C(\varepsilon_0)$ such that, for any density $f$ with respect to $\nu^n_{\rho(\cdot)}$ and any $\delta>0$,
\begin{equation*}
\begin{split}
&\int  \sum_{x\in\bb T_n}w(x) w(x+1) G(x)f d\nu^n_{\rho(\cdot)}\\
\leq\,&\delta n^2 D(\sqrt{f};\nu^n_{\rho(\cdot)})\,+\, \frac{C(1+\kappa^2)}{\delta}\big(\|G\|_\infty\,+\,\|G\|_\infty^2\big)\big(H(f;\nu^n_{\rho(\cdot)})\,+\, 8\big),
\end{split}
\end{equation*}
where $H(f;\nu^n_{\rho(\cdot)}):=\int f\log f d\nu^n_{\rho(\cdot)}$.
\end{lemma}
In the proof of this lemma in \cite{jm}, the authors defined an important operator $U_\delta$  which acts on functions $G:\bb T_n\to\bb R$. The following estimates obtained there on $U_\delta$ will be useful later:
\begin{equation}\label{v2}
\int  \sum_{x\in\bb T_n}w(x) w(x+1) G(x)f d\nu^n_{\rho(\cdot)}\,\leq\,\delta n^2D(\sqrt{f};\nu^n_{\rho(\cdot)})\,+\,\int U_\delta(G) fd\nu^n_{\rho(\cdot)}
\end{equation}
and 
\begin{equation}\label{sum2}
\int |U_\delta(G)|f d\nu^n_{\rho(\cdot)}\,\leq\, C(\varepsilon_0,\kappa)\frac{\|G\|_\infty+\|G\|_\infty^2}{\delta}\big(H(f;\nu^n_{\rho(\cdot)})\,+\,8\big).
\end{equation}
We refer to the original proof of Lemma 3.1 in \cite{jm} for the explicit expression of $U_\delta(\cdot)$ and other details.


\subsection{Replacement lemma}\label{sec33}

Before going into the replacement lemma, we need to state some known results. The first one connects tail bounds and moments estimates.
\begin{lemma}\label{pe}[Lemma 4.3 of \cite{jm}]
Let $X$ be a nonnegative random variable. Assume there exists a constant $C_0>0$ such that $P(X>\Delta)\,\leq\,C_0/\Delta^2$ for all $\Delta>0$. Then  for any $\lambda\in(1,2)$, there exists a constant $C(\lambda)$ independent of $X$,  such that $E[X^\lambda]\,\leq\, C(\lambda) C_0^{\lambda/2}$.
\end{lemma}
The second one is a well known consequence of the entropy inequality.
\begin{proposition}\label{entpro}[Proposition 8.2 in Appendix 1 of \cite{kl}]
Let $\mu$ be a measure on a finite space $\Omega $ and let $f$ be a density with respect to $\mu$. Then for any $A\subset \Omega$,
$$\int_A fd\mu\,\leq\,\frac{H(f;\mu)\,+\,\log 2}{\log \big(1+\mu(A)^{-1}\big)}.$$
\end{proposition}

We now turn to the replacement lemma. Recall the definition \eqref{Ztnh} of $Z_{t,n}^{h,\varepsilon}$. If one tries to replace $\Gamma_n^h(t)$ by $Z_{t,n}^{h,\varepsilon}$, there will be some error terms coming out whose moment estimates are complicated to deal with. We need to take these error terms out of $\Gamma_n^{h}(t)$. Thus, define the errror process $\{J_n^{h,\varepsilon}(t):t\in[0,T]\}$ by 
$$J_n^{h,\varepsilon}(t)\,:=\,J_{n,1}^{h}(t)\,+\,J_{n,2}^{h,\varepsilon}(t)\,+\,J_{n,3}^{h,\varepsilon}(t)$$
where $J_{n,1}^{h}(t),J_{n,2}^{h,\varepsilon}(t), J_{n,3}^{h,\varepsilon}(t)$ are given below in \eqref{defJ1},\eqref{defJ2} and \eqref{defJ3} respectively. It will be clear from the proof of the local replacement lemma why the error process $J_n^{h,\varepsilon}(t)$ are defined in such way.

The following process will be used to perform the replacement:
$$\widetilde{\Gamma}_n^{h,\varepsilon}(t)\,:=\, \Gamma_n^h(t)\,-\,J_n^{h,\varepsilon}(t).$$ 
The next lemma shows that $Z_{t,n}^{h,\varepsilon}$ and $\widetilde{\Gamma}^{h,\varepsilon}_n(t)$ are close under the $L^{\lambda}$ norm for $1<\lambda<2$.
\begin{lemma}[Local Replacement Lemma]\label{ZG}
Fix $0\leq s < t \leq T$. Then for every $\lambda\in (1,2)$, there exists a positive constant $C$ independent of $n$ and $\varepsilon$ such that
$$\bb E_{\mu^n}\Big[\big|Z_{t,n}^{h,\varepsilon}\,-\,Z_{s,n}^{h,\varepsilon}\,-\,\big(\widetilde{\Gamma}^{h,\varepsilon}_n(t)-\widetilde{\Gamma}^{h,\varepsilon}_n(s)\big)\big|^\lambda\Big]\,\leq\, C(t-s)^{\frac{\lambda}{2}}\varepsilon^{\frac{\lambda}{2}} $$
for every $n\in \bb N$.
\end{lemma}
\begin{proof}
It is enough to prove the lemma only for $s<t$ such that $t-s\leq 1$. Thanks to Lemma \ref{pe}, to prove this lemma,  it is enough to show that, there exists a constant $C>0$ such that
$$\bb P_{\mu^n}\Big[\big\lvert Z_{t,n}^{h,\varepsilon}\,-\,Z_{s,n}^{h,\varepsilon}\,-\,\big(\widetilde{\Gamma}^{h,\varepsilon}_n(t)-\widetilde{\Gamma}^{h,\varepsilon}_n(s)\big)\big\rvert>\Delta\Big]\,\leq\, \frac{C(t-s)\varepsilon }{\Delta^2},$$
for all $\Delta>0$.

Recall that 
$$\overline{\eta}_{rn^2}(A)=\prod_{x\in A}\Big\{\eta_{rn^2}(x)-\rho_r(\frac{x}{n})\Big\}, \quad \text{for every subset}\, A\subset\bb Z.$$ 
Define
$$V^\varepsilon_{r,z}(\eta):=\sqrt{n}\Big\{\frac{\overline{\eta}(z)}{q_r}\,-\,\frac{1}{n}\sum_{x=1}^{n}\phi_\varepsilon\big(\frac{x}{n}\big) w(x)\frac{\mc X(\rho_r(0))}{q_r}\Big\}.$$
This corresponds replacement for $\overline{\eta}(A)$ with $|A|=1$. Given a set $A\subset\bb Z$ with $|A|\geq 2$, denote by $A_l$ the smallest element of $A$ and by $A_r$ the largest element of $A$. Let us define
\begin{equation}\label{defW}
\begin{split}
W_{r,A}^{\varepsilon}(\eta)\,=\, &\sqrt{n}\Big\{\frac{\overline{\eta}(A)}{q_r}\,-\,\Big(\frac{1}{n^{1/6}}\sum_{x=1}^{n^{1/6}} w(-x+A_l)\mc X(\rho_r\big(\frac{A_l}{n}\big))\Big)\\
& \times\frac{\overline{\eta}(A\backslash\{A_l,A_r\})}{q_r}  \Big(\frac{1}{n}\sum_{x=1}^{n}\phi_\varepsilon\big(\frac{x}{n}\big) w(x+A_r)\mc X(\rho_r\big(\frac{A_r}{n}\big))\Big)\Big\}.
\end{split}
\end{equation}
To simplify the notation, denote the long product term after the minus sign by $Q_{r,A}^{\varepsilon}(\eta)$. One may already see from the way we define $W_{r,A}^{\varepsilon}$ that, for $\overline{\eta}_{rn^2}(A)$ with $|A|\geq 2$, replacements will be performed on $\overline{\eta}_{rn^2}(A_l)$ and $\overline{\eta}_{rn^2}(A_r)$ only.  

By notation above and \eqref{dech}, we have
\begin{equation}\label{dif4terms}
\begin{split}
&\big(\Gamma^h_n(t)\,-\,\Gamma^h_n(s)\big)\,-\,(Z_{t,n}^{h,\varepsilon}\,-\,Z_{s,n}^{h,\varepsilon})\\
=\,&  \int_s^t \sum_{|z|\leq R}c_r(\{z\})V_{r,z}^\varepsilon dr\,+\, \int_s^t \sum_{\substack{A\subset[-R,R]\cap\bb Z\\|A|\geq 2}} c_r(A)W_{r,A}^{\varepsilon}dr\\
+\,& \big(J_{n,1}^{h}(t) - J_{n,1}^{h}(s)\big)\,+\, \big(J_{n,2}^{h,\varepsilon}(t) - J_{n,2}^{h,\varepsilon}(s)\big)\,+\,\big(J_{n,3}^{h,\varepsilon}(t) - J_{n,3}^{h,\varepsilon}(s)\big)
\end{split}
\end{equation}
where 
\begin{equation}\label{defJ1}
J_{n,1}^{h}(t)\,=\, \sqrt{n}\int_0^t \frac{c_s(\varnothing)\,-\,\varphi_h(\rho_s(0))}{q_s}ds
\end{equation}
 \begin{equation}\label{defJ2}
J_{n,2}^{h,\varepsilon}(t)\,=\, \sqrt{n}\int_0^t \Big[ \sum_{|z|\leq R}c_r(\{z\})\,-\,\varphi'_h(\rho_s(0))\Big]\frac{1}{n}\sum_{x=1}^{n}\phi_\varepsilon\big(\frac{x}{n}\big) w_{rn^2}(x)\frac{\mc X(\rho_r(0))}{q_r}dr
\end{equation}
 \begin{equation}\label{defJ3}
J_{n,3}^{h,\varepsilon}(t)\,=\, \sqrt{n}\int_0^t \sum_{\substack{A\subset[-R,R]\cap\bb Z\\|A|\geq 2}} c_r(A)Q_{r,A}^{\varepsilon}dr
\end{equation}
Therefore, to prove the lemma, it is sufficient to bound the probability that the absolute value of the first two terms at the right hand side of \eqref{dif4terms} is larger than $\Delta$, from above by $C(t-s)\varepsilon \Delta^{-2}$. Those estimates are achieved in Lemma \ref{termone} and Lemma \ref{termtwo} respectively in the next two subsections. 
\end{proof}

\subsection{Estimate of the first term in \eqref{dif4terms}}\label{sec34}
\begin{lemma}\label{termone}
There exists a constant $C$ independent of $n$ and $\varepsilon$ such that 
$$\bb P_{\mu^n}\Big[ \Big\lvert \int_s^t \sum_{|z|\leq R}c_r(\{z\})V_{r,z}^\varepsilon ds\Big\rvert>\Delta\Big]\,\leq\, \frac{C(t-s)\varepsilon }{\Delta^2}$$
for all $\Delta>0$.
\end{lemma}
\begin{proof}
Let us denote
$$V_r^\varepsilon\,=\, \sum_{|z|\leq R}c_r(\{z\})V_{r,z}^\varepsilon.$$
 By triangle inequality, to prove the lemma, it is sufficient to show that 
\begin{equation}\label{eqtone}
\bb P_{\mu^n}\Big[\int_s^t \pm V_r^\epsilon \,-\,\frac{\gamma}{2}\, U_\delta(G_r) dr >\Delta\Big]\,\leq\, \frac{C(t-s)\varepsilon }{\Delta^2}
\end{equation}
\begin{equation}\label{eqttwo}
\bb P_{\mu^n}\Big[\Big\lvert\int_s^t \frac{\gamma}{2}\, U_\delta(G_r)dr\Big\rvert>\Delta\Big]\,\leq\,  \frac{C(t-s)\varepsilon }{ \Delta^2}
\end{equation}
for some constants $\gamma>0$, $\delta>0$ and a function $G_r:\bb T_n\to \bb R$ that will be determined later. Here $U_\delta$ is the operator in subsection \ref{subsecEnt} satisfying \eqref{v2}\eqref{sum2}.

The second inequality \eqref{eqttwo} is easy to deal with. Choose $\delta=1/4$ and write $U_{1/4}(G)$ simply as $U(G)$ in the rest of this section. By Markov inequality, the left hand side of \eqref{eqttwo} is bounded by 
$$\Delta^{-1}\bb E_{\mu^n}\Big[\Big\lvert\int_s^t\frac{\gamma}{2}\, U(G_r)dr\Big\rvert\Big]\,\leq\, \frac{(t-s)\gamma}{2\Delta}\sup_{0\leq r\leq T}\bb E_{\mu^n}\big[\big\lvert U(G_r)\big\rvert\big].$$
Since $H_n(r)$ is of order $O(1)$ for every $0\leq r\leq T$, by  \eqref{sum2}, as long as $\|G_r\|_\infty$ is uniformly bounded over $n$ and $r$, $\bb E_{\mu^n}\big[\big\lvert U(G_r)\big\rvert\big]$ is bounded uniformly over $r$ by some constant $C$.  Choosing $\gamma=B(t-s)\varepsilon\Delta^{-1}$ with $B=B(\varepsilon_0, \phi, M)$ that will be determined later, since $t-s\leq1$, inequality \eqref{eqttwo} is proved.

We now turn to prove the first inequality \eqref{eqtone}. By Proposition \ref{entpro},
$$\bb P_{\mu^n}\Big[\int_s^t \pm V_r^\varepsilon\,-\,\frac{\gamma}{2} U(G_r) dr >\Delta\Big]\,\leq\,\frac{\log 2}{\log\Big(1\,+\, \bb P_{\mu^n}\Big[\int_s^t \pm V_r^\varepsilon\,-\,\frac{\gamma}{2}\, U(G_r) dr>\Delta\Big]^{-1} \Big)}.$$
Thus it is enough to show that
\begin{equation}\label{midexp}
\bb P_{\mu^n}\Big[\int_s^t \pm V_r^\varepsilon\,-\,\frac{\gamma}{2}\, U(G_r) dr>\Delta\Big]\,\leq\, \exp\Big\{\frac{-\Delta^2}{C(t-s)\varepsilon}\Big\}.
\end{equation}
By Markov inequality,
\begin{equation}\label{expE}
\bb P_{\mu^n}\Big[\int_s^t \pm V_r^\varepsilon\,-\,\frac{\gamma}{2}\, U(G_r) dr>\Delta\Big]\,\leq\, e^{-\frac{\Delta}{\gamma}}\bb E_{\mu_n}\Big[\exp\Big\{\int_s^t \pm \frac{V_r^\varepsilon}{\gamma}\,-\,\frac{U(G_r)}{2} dr\Big\}\Big].
\end{equation}
If we can prove
\begin{equation}\label{expbound}
\begin{split}
&\bb E_{\mu^n}\Big[\exp\Big\{\int_s^t \pm \frac{V_r^\varepsilon}{\gamma}\,-\,\frac{U(G_r)}{2} dr\Big\}\Big]\\
\leq\, &\exp\Big\{ \frac{4\|\phi\|_\infty^2 M^2}{\varepsilon_0} \frac{\varepsilon(t-s)}{\gamma^2}\sup_{0\leq t\leq T}\big(\varphi'(\rho_t(0))\big)^2  \Big\},
\end{split}
\end{equation}
then we claim that \eqref{midexp} holds by choosing 
$$B=8\|\phi\|_\infty^2 M^2\varepsilon_0^{-1}\sup_{0\leq t\leq T}\big(\varphi'(\rho_t(0))\big)^2.$$
 Indeed, since we have chosen $\gamma=B(t-s)\varepsilon\Delta^{-1}$, a direct computation gives
$$\frac{4\|\phi\|_\infty^2 M^2}{\varepsilon_0} \frac{\varepsilon(t-s)}{\gamma^2}\sup_{0\leq t\leq T}\big(\varphi'(\rho_t(0))\big)^2\,= \,\frac{\Delta}{2\gamma},$$
thus the right hand side of \eqref{expE} becomes $\exp\{-\Delta/2\gamma\}$.

It remains to prove \eqref{expbound}. By Lemma A.2 of \cite{jm} and inequality \eqref{relD}, the logarithm of the expectation at the left hand side of \eqref{expbound} is bounded by 
\begin{equation}\label{stbound}
\begin{split}
\int_s^t \Big\{\sup_{f} &\frac{-n^2}{2}D(\sqrt{f};d\mu_r^n)+\int\frac{\pm V^\varepsilon_r}{\gamma}fd\mu_r^n\\
&+\,\frac{1}{2} \int\big(L_{n,r}^\star 1\,-\,\frac{d}{dr}\log\Psi_r^n-U(G_r)\big)f d\mu_r^n \Big\}dr,
\end{split}
\end{equation}
where the supremum is taken over all the density $f$ with respect to the reference measure $\mu_r^n$, $L_{n,r}^\star$ is the adjoint operator of $L_n$ with respect to $\mu_r^n$ and $\Psi_r^n$ is the Radon-Nikydim derivative of $\mu_r^n$ with respect to $\nu^n_{1/2}$.

In section A.3 of \cite{jm} it is shown that
\begin{equation*}
\Big(L_{n,r}^\star 1\,-\,\frac{d}{dr}\log\Psi_r^n\Big)(\eta)\,=\, \frac{1}{n^2}\sum_{x\in\bb T_n} w(x) R_r^n(x)\,+\, \sum_{x\in\bb T_n}w(x) w(x+1) G^n_r(x)
\end{equation*}
where $R_r^n:\bb T_n\to\bb R$ is a function such that 
$$\sup_{x\in\bb T_n}|R_r^n(x)|\,\leq\,\sup_{u\in\bb T} \frac{d}{du^4}\rho_r(u)$$
 and 
\begin{equation*}
\begin{split}
G^n_r(x)\,=\,&n\Big[\rho_r\big(\frac{x+1}{n}\big)-\rho_r\big(\frac{x}{n}\big)\Big] F\big(\frac{x}{n}\big)\Big[\rho_r\big(\frac{x+1}{n}\big)+\rho_r\big(\frac{x}{n}\big)-2\rho_r\big(\frac{x+1}{n}\big)\rho_r\big(\frac{x}{n}\big)\Big]\\
 -\,&n^2\Big[\rho_r\big(\frac{x+1}{n}\big)\,-\,\rho_r\big(\frac{x}{n}\big)\Big]^2.
\end{split}
\end{equation*}
Clearly $\|G^n_r\|_\infty$ is bounded uniformly over $n$ and $r\in[0,T]$. This function $G^n_r$ is our choice of $G_r:\bb T_n\to\bb R$ anticipated in \eqref{eqtone} and \eqref{eqttwo}. 

For every density $f$ with respect to $\mu_r^n$, since $|w(x)|\,\leq C(\varepsilon_0)$ and $\sup_{u\in\bb T} \frac{d}{du^4}\rho_r(u)$ is uniformly bounded, 
$$\frac{1}{n^2}\int_s^t\int \sum_{x\in\bb T_n} w(x) R_r^n(x) fd\mu_r^ndr\,=\,O(n^{-1}).$$
 On the other hand, by inequality \eqref{v2}, since $\delta$ was chosen to be $1/4$,
$$\int \sum_{x\in\bb T_n}w(x) w(x+1)G^n_r(x) fd\mu_r^n\,\leq\,\int U(G_r)fd\mu_r^n\,+\,\frac{n^2}{4} D(\sqrt{f};\mu_r^n)$$
 for any density $f$ with respect to the measure $\mu_r^n$. 

Estimates above allow us to conclude that
\begin{equation*}
\int_s^t \sup_{f}\Big\{-\frac{n^2}{4}D(\sqrt{f};d\mu_r^n)+\,\frac{1}{2} \int\big(L_{n,r}^\star 1\,-\,\frac{d}{dr}\log\Psi_r^n-U(G_r)\big)f d\mu_r^n \Big\}dr
\end{equation*}
vanishes as $n\to\infty$.
From Corollary \ref{degonerep} by choosing $b=\pm \gamma^{-1}$, since $ q_r\geq M^{-1}$ for all $0\leq r\leq T$, we get 
\begin{equation*}
\begin{split}
&\int_s^t \sup_{f} \Big\{ \frac{-n^2}{4}D(\sqrt{f};d\mu_r^n)\,+\, \int\frac{\pm V^\varepsilon_r}{\gamma}fd\mu_r^n \,\Big\}dr\\
\leq\, &\frac{4\|\phi\|_\infty^2 M^2}{\varepsilon_0} \frac{\varepsilon(t-s)}{\gamma^2}\sup_{0\leq t\leq T}\big(\varphi'(\rho_t(0))\big)^2.
\end{split}
\end{equation*}
This proves \eqref{expbound} and hence the lemma.
\end{proof}

\begin{remark}
We applied Lemma A.2 of \cite{jm} to obtain the upper bound stated in \eqref{stbound}. The original statement of that Lemma is limited to the setting $s=0$. It is easy to extend the lemma to $s>0$, which is actually needed in the proof of the previous Theorem. We just need to consider a time-shifted version of the original Markov chain:
$$\xi_{tn^2}(\cdot)=\eta_{(t+s)n^2}(\cdot), \quad \text{for all}\,\, t\geq 0,$$
and then apply Lemma A.2 of \cite{jm} to the new process \{$\xi_{tn^2}: 0\leq t\leq T-s\}$.
\end{remark}

\begin{lemma}\label{dif}
 Let $\nu^n_{\rho(\cdot)}$ be the Bernoulli product measure with slowly  varying parameter associated to a smooth profile $\rho(\cdot)$ which takes value in $(\varepsilon_0,1-\varepsilon_0)$.  Assume that $f$ is density with respect to $\nu_{\rho(\cdot)}^n$ such that $H(f;\nu_{\rho(\cdot)}^n)=o(\sqrt{n})$. Fix $a\in\bb R$, $q>0$, $\varepsilon\in(0,1)$ and a function $g:\Omega_n\to [-1,1]$ whose support is disjoint with $[0,\varepsilon n]\cap\bb T_n$. Then for every density $f$ with respect to $\nu_{\rho(\cdot)}^n$ we have
\begin{equation*}
\begin{split}
\limsup_{n\to\infty}\Big\{&\int  a\sqrt{n}\Big\{\frac{\overline{\eta}(0)}{q}-\frac{1}{n}\sum_{x=1}^{n}\phi_\varepsilon\big(\frac{x}{n}\big) w(x)\frac{\mc X(\rho(0))}{q}\Big\}gf d\nu_{\rho(\cdot)}^n\,-\, \frac{n^2}{8}D(\sqrt{f};\nu_{\rho(\cdot)}^n)\Big\}\\
&\leq\,\frac{4a^2\varepsilon \|\phi\|_\infty^2}{\varepsilon_0 q^2}.
\end{split}
\end{equation*}

\end{lemma}
\begin{proof}
Since $\phi$ is smooth and has integral $1$, applying a Taylor expansion,  one can obtain that 
$$\Big|1-\frac{1}{n}\sum_{x=1}^n\phi_\varepsilon\big(\frac{x}{n}\big)\Big|\,\leq\,\frac{\|\phi'\|_\infty}{\varepsilon n}.$$
This combining the assumption $|g|\leq 1$ implies that
$$\int  a\sqrt{n}\Big\{\frac{\overline{\eta}(0)}{q}-\frac{1}{n}\sum_{x=1}^{n}\phi_\varepsilon\big(\frac{x}{n}\big) w(x)\frac{\mc X(\rho(0))}{q}\Big\}gf d\nu_{\rho(\cdot)}^n$$ is less than or equal to 
\begin{equation}\label{longsum}
\int \frac{a\mc X(\rho(0))\sqrt{n}}{qn}\sum_{x=1}^{ n-1}\phi_\varepsilon\big(\frac{x}{n}\big)\Big\{w(0)\,-\, w(x)\Big\}gf d\nu_{\rho(\cdot)}^n
\end{equation}
added with an error term $\frac{a\sqrt{n}\|\phi'\|_\infty^2}{\varepsilon nq}$ which vanishes as $n\to\infty$. Recall that $\phi_\varepsilon$ has a compact support in $(0,\varepsilon)$. Writting each $w(0)\,-\, w(x)$ as a telescope sum, the expression in \eqref{longsum} becomes 
$$\int \frac{a\mc X(\rho(0))\sqrt{n}}{q n}\sum_{x=0}^{\varepsilon n}\sum_{y=x+1}^{\varepsilon n}\phi_\varepsilon\big(\frac{y}{n}\big)\big(w(x)-w(x+1)\big) gf d\nu_{\rho(\cdot)}^n.$$
 By Lemma E.3 in \cite{jm}, since $|g|\leq 1$ and $g$ is not supported on $[0,\varepsilon n]$,     for any $\delta_x>0$, $x\in\bb T_n$, the previous expression is bounded by 
\begin{equation}\label{sumA}
\begin{split}
\frac{\mc X(\rho(0))\sqrt{n}}{q n}&\sum_{x=0}^{ \varepsilon n}\sum_{y=x+1}^{\varepsilon n}\phi_\varepsilon\big(\frac{y}{n}\big)\Big\{ \delta_x n^2D_{x,x+1}(\sqrt{f};\nu_{\rho(\cdot)}^n)\,+\, \frac{4a^2}{\delta_x \varepsilon_0 n^2}\\
\,-\, &a\Big[\rho\big(\frac{x+1}{n}\big)-\rho\big(\frac{x}{n}\big)\Big]\int w(x)w(x+1)gfd\nu_{\rho(\cdot)}^n\Big\}.
\end{split}
\end{equation}
In the above formula, $D_{x,x+1}(\sqrt{f};\nu_{\rho(\cdot)}^n)$ is the Dirichelet form of $\sqrt{f}$ with respect to $\nu_{\rho(\cdot)}^n$ corresponding to the jumps between sites $x$ and $x+1$,
$$D_{x,x+1}(\sqrt{f};\nu_{\rho(\cdot)}^n)\,=\, \sum_{\eta\in\Omega_n}\big[\sqrt{f(\eta^{x,x+1})}\,-\,\sqrt{f(\eta)}\big]^2 \nu_{\rho(\cdot)}^n.$$

Taking 
$$\delta_x=\frac{qn}{16 \mc X(\rho(0))\sqrt{n}\sum_{y=x+1}^{\varepsilon n+1}\phi_\varepsilon\big(\frac{y}{n}\big)}\,,$$ 
using the fact 
$$\sum_{y=1}^{\varepsilon n}\phi_{\varepsilon}\big(\frac{y}{n}\big)\,\leq\, n\|\phi\|_\infty,$$
 the sum of the first two terms in \eqref{sumA} is less than or equal to 
$$\frac{n^2}{16}D(\sqrt{f};\nu_{\rho(\cdot)}^n)\,+\,\frac{64\varepsilon a^2\big(\mc X(\rho(0))\big)^2\|\phi\|_\infty^2}{q^2\varepsilon_0}.$$ 
On the other hand, choosing 
$$G(x)\,=\,-\mathds{1}_{\{0\leq x\leq \varepsilon n\}}sgn\Big(a\Big[\rho\big(\frac{x+1}{n}\big)-\rho\big(\frac{x}{n}\big)\Big]\Big)$$
and $\delta=\frac{1}{8}$ in Lemma \ref{maing}, we have 
$$\sum_{x=0}^{\varepsilon n} G(x)\int w(x)w(x+1)gfd\nu_{\rho(\cdot)}^n\,\leq\, \frac{n^2}{8} D(\sqrt{f};\nu_{\rho(\cdot)}^n)\,+\, C\big(H(f;\nu_{\rho(\cdot)}^n)\,+\, 8\big).$$
Therefore the third term in \eqref{sumA}
\begin{equation*}
\begin{split}
-&\frac{a\mc X(\rho(0))\sqrt{n}}{q n}\sum_{x=0}^{ \varepsilon n}\sum_{y=x+1}^{\varepsilon n}\phi_\varepsilon\big(\frac{y}{n}\big)\Big[\rho\big(\frac{x+1}{n}\big)-\rho\big(\frac{x}{n}\big)\Big]\int w(x)w(x+1)gfd\nu_{\rho(\cdot)}^n\\
\leq\, &\frac{\lvert a\rvert\mc X(\rho(0))\sqrt{n}}{q n}\frac{\kappa n\|\phi\|_\infty}{n}\Big\{ \frac{1}{8} n^2 D(\sqrt{f};\mu_t^n)\,+\, C\big(H(f;\nu_{\rho(\cdot)}^n)\,+\, 8\big)\Big\}\\
\leq\, &\frac{\lvert a\rvert\mc X(\rho(0)) C(\kappa)\|\phi\|_\infty}{q\sqrt{n}} \big\{n^2D(\sqrt{f};\nu_{\rho(\cdot)}^n)\,+\, H(f;\nu_{\rho(\cdot)}^n)\,+\,8\}.
\end{split}
\end{equation*}

In summary, we have proved that 
\begin{equation*}
\begin{split}
&a\sqrt{n}\Big\{\frac{\overline{\eta}(0)}{q}-\frac{1}{n}\sum_{x=1}^{n}\phi_\varepsilon\big(\frac{x}{n}\big) w(x)\frac{\mc X(\rho(0))}{q}\Big\}gf d\nu_{\rho(\cdot)}^n\,-\, \frac{n^2}{8}D(\sqrt{f};\nu_{\rho(\cdot)}^n)\\
\leq\,&\Big(\frac{a\mc X(\rho(0)) C(\kappa)\|\phi\|_\infty}{q\sqrt{n}}\,-\,\frac{1}{16}\Big) n^2D(\sqrt{f};\nu_{\rho(\cdot)}^n)\,+\,\frac{64\varepsilon a^2\big(\mc X(\rho(0))\big)^2\|\phi\|_\infty^2}{q^2\varepsilon_0}\\
&\,+\,\frac{a\mc X(\rho(0)) C(\kappa)\|\phi\|_\infty}{q\sqrt{n}} \big\{H(f;\nu_{\rho(\cdot)}^n)\,+\,8\}\,+\,\frac{a\sqrt{n}\|\phi'\|_\infty^2}{\varepsilon nq}.
\end{split}
\end{equation*}
Under the assumption $H(f;\mu_t^n)=o(\sqrt{n})$ and by the fact $\mc X(\rho(0))\leq \frac{1}{4}$, taking the limit $n\to\infty$, the first, third and fourth terms at the right hand side of the above inequality either vanish or become negative, thus we finish the proof.
\end{proof}
\begin{corollary}\label{degonerep}
Fix $T>0$ and $b\in\bb R$. Then for every $0\leq t\leq T$, every density $f$ with respect to $\mu_t^n$, we have
\begin{equation*}
\limsup_{n\to\infty}\Big\{b\int V_t^\varepsilon f d\mu_t^n\,-\, \frac{n^2}{4}D(\sqrt{f};\mu_t^n)\Big\}\leq\,\frac{4\varepsilon b^2\|\phi\|_\infty^2\big(\varphi'(\rho_t(0))\big)^2}{\varepsilon_0 q_t^2}.
\end{equation*}
\end{corollary}
\begin{proof}
Recall that
$$V^\varepsilon_t(\eta_{tn^2})\,=\,\sum_{|z|\leq R} c_t(\{z\})\sqrt{n}\Big\{\frac{\overline{\eta}_{tn^2}(z)}{q_t}\,-\,\frac{1}{n}\sum_{x=1}^{n}\phi_\varepsilon\big(\frac{x}{n}\big) w_{tn^2}(x)\frac{\mc X(\rho_t(0))}{q_t}\Big\}.$$
Taking $g=1$ and $a=b\sum_{|z|\leq R} c_t(\{z\})$ in Lemma \ref{dif}, in view of claim \eqref{claimdec}, we have  
\begin{equation*}
\begin{split}
\limsup_{n\to\infty}\Big\{b\int  &\sum_{|z|\leq R} c_t(\{z\})\sqrt{n}\Big\{\frac{\overline{\eta}_{tn^2}(0)}{q_t}-\frac{1}{n}\sum_{x=1}^{n}\phi_\varepsilon\big(\frac{x}{n}\big) w_{tn^2}(x)\frac{\mc X(\rho(0))}{q_t}\Big\}f d\mu_t^n\\
-\, &\frac{n^2}{8}D(\sqrt{f};\mu_t^n)\Big\}\,\leq\,\frac{4 \varepsilon b^2\|\phi\|_\infty^2\big(\varphi'(\rho_t(0))\big)^2}{\varepsilon_0 q_t^2}.
\end{split}
\end{equation*}
To conclude the corollary, it remains to show that
\begin{equation}\label{centerest}
\limsup_{n\to\infty}\Big\{b\int \sum_{|z|\leq R} c_t(\{z\})\sqrt{n}\Big\{\frac{\overline{\eta}_{tn^2}(z)}{q_t}\,-\,\frac{\overline{\eta}_{tn^2}(0)}{q_t}\Big\} f d\mu_t^n\,-\,\frac{n^2}{8}D(\sqrt{f};\mu_t^n)\Big\}\,=\,0.
\end{equation}

Since $\rho_t$ is smooth, after a Taylor expansion to $\rho_t(\cdot)$, we see 
\begin{equation}\label{valueest}
 \int b\sum_{|z|\leq R} c_t(\{z\})\sqrt{n}q_t^{-1}\Big\{ \rho_t\big(\frac{z}{n}\big)\,-\, \rho_t (0)\Big\} f d\mu_t^n 
\end{equation}
is of order $O(n^{-1/2})$, thus it vanishes as $n\to\infty$. 

On the other hand, for every $x\in\bb Z$, a change of varaibles gives
\begin{equation}\label{changev}
\begin{split}
&\int [\eta(x+1)-\eta(x)]f(\eta) d\mu_t^n\\
=\, &\int \eta(x)[f(\eta^{x,x+1})-f(\eta)]d\mu_t^n\,+\, \int \eta(x)f(\eta^{x,x+1})\Big[ \frac{\mu_t^n(\eta^{x,x+1})}{\mu_t^n(\eta)}\,-\,1\Big] d\mu_t^n.
\end{split}
\end{equation}
Since $\mu_t^n$ was defined as the Bernoulli product measure with slowing varying smooth profile $\rho_t(\cdot)$,  a direct compuation shows that
\begin{equation}\label{rn}
\Big|  \frac{\mu_t^n(\eta^{x,x+1})}{\mu_t^n(\eta)}\,-\,1\Big|\,\lesssim\, \frac{1}{n}.
\end{equation}
This implies that the second term at the right hand side of \eqref{changev}  is bounded by 
$$\frac{C}{n} \int f(\eta^{x,x+1})d\mu_t^n\,=\, \frac{C}{n} \int f(\eta)  \frac{\mu_t^n(\eta^{x,x+1})}{\mu_t^n(\eta)}d\mu_t^n\,\leq\, \frac{C}{n}.$$
Using elementary inequality 
$$a-b\,\leq\,A(\sqrt{a}-\sqrt{b})^2\,+\, \frac{(\sqrt{a}+\sqrt{b})^2}{4A}\,\leq\, A(\sqrt{a}-\sqrt{b})^2\,+\, \frac{a+b}{2A}$$
for any $a,b,A>0$, by \eqref{rn}, we estimate the first term at the right hand side of \eqref{changev} by 
$$AD_{x,x+1}(\sqrt{f};d\mu_t^n)\,+\, \frac{2+C/n}{2A},\quad \forall \,\,A>0.$$
It follows from estimates above that 
$$\int [\eta(x+1)-\eta(x)]f(\eta) d\mu_t^n \,\leq\,  AD_{x,x+1}(\sqrt{f};d\mu_t^n)\,+\,\frac{C}{A}.$$
Writting $\eta(z)-\eta(0)$ as a telescope sum $\sum_{x=0}^{z-1}[\eta(x+1)-\eta(x)]$, using the previous estimate, we have
$$b\int \sum_{|z|\leq R} c_t(\{z\})\sqrt{n}\Big\{\frac{\eta_{tn^2}(z)}{q_t}\,-\,\frac{\eta_{tn^2}(0)}{q_t}\Big\} f d\mu_t^n\,\leq\, AD(\sqrt{f};\mu_t^n)\,+\, \frac{Cn}{A},$$
where the constant $C$ depends on $b$ and $c_t(\{z\})$ with $|z|\leq R$. Choosing $A=n^{3/2}$, we can conclude that 
\begin{equation}\label{variableest}
\limsup_{n\to\infty}\Big\{b\int \sum_{|z|\leq R} c_t(\{z\})\sqrt{n}\Big\{\frac{\eta_{tn^2}(z)}{q_t}\,-\,\frac{\eta_{tn^2}(0)}{q_t}\Big\} f d\mu_t^n\,-\,\frac{n^2}{8}D(\sqrt{f};\mu_t^n)\Big\}\,=\,0.
\end{equation}

\eqref{centerest} follows from \eqref{valueest} and \eqref{variableest} and we finish the proof.

\end{proof}

\subsection{Estimate of the second term in \eqref{dif4terms}}\label{sec35}
\begin{lemma}\label{termtwo}
There exists a constant $C$ independent of $n$ and $\varepsilon$ such that 
$$\bb P_{\mu^n}\Big[ \Big\lvert  \int_s^t \sum_{\substack{A\subset[-R,R]\cap\bb Z\\|A|\geq 2}} c_r(A)W_{r,A}^{\varepsilon}dr\Big\rvert>\Delta\Big]\,\leq\, \frac{C(t-s)\varepsilon }{\Delta^2}$$
for all $\Delta>0$.
\end{lemma}
\begin{proof}
Note that there are finitely many subsets $A\subset [-R,R]\cap \bb Z$ such that $|A|\geq 2$ because $R$ is finite. Since $c_r(A)$ is uniformly bounded over all $r\in[0,T]$, it is enough to show that 
$$\bb P_{\mu^n}\Big[ \Big\lvert \int_s^t c_r(A)W_{r,A}^{\varepsilon}\,dr\Big\rvert>\Delta\Big]\,\leq\, \frac{C(t-s)\varepsilon }{\Delta^2}$$
for every $A\subset [-R,R]\cap \bb Z$ such that $|A|\geq 2.$
Similar to what we did in Lemma \ref{termone}, we are going to show that
\begin{equation}\label{degtwoeqtone}
\bb P_{\mu^n}\Big[\int_s^t \pm c_r(A)W_{r,A}^{\varepsilon}\,-\,\frac{\gamma}{2}\, U(G_r) dr >\Delta\Big]\,\leq\, \frac{C(t-s)\varepsilon }{\Delta^2}
\end{equation}
\begin{equation}\label{degtwoeqttwo}
\bb P_{\mu^n}\Big[\Big\lvert\int_s^t \frac{\gamma}{2}\, U(G_r)dr\Big\rvert>\Delta\Big]\,\leq\,  \frac{C(t-s)\varepsilon }{ \Delta^2}
\end{equation}
for some constants $\gamma>0$, and a function $G_r:\bb T_n\to \bb R$ to be chosen later.

 It was shown in the proof of Lemma \ref{termone} that, \eqref{degtwoeqttwo} holds if $\|G_r\|_\infty$ is uniformly bounded over $r\in[0,T]$ and $\gamma=B(t-s)\varepsilon\Delta^{-1}$ for some constant $B$ that will be determined later. For the proof of \eqref{degtwoeqtone}, in view of how we dealt with inequality \eqref{eqtone}, it is enough to show that 
\begin{equation}\label{intcW}
\begin{split}
&\int_s^t \sup_{f}\Big\{-\frac{n^2}{4}D(\sqrt{f};d\mu_r^n)\,+\, \int\pm c_r(A)W_{r,A}^{\varepsilon} d\mu_r^n \Big\}dr\\
\leq\, &\frac{4(t-s)\varepsilon \|\phi\|_\infty^2M^2}{\varepsilon_0 }\Big(\sup_{0\leq r\leq T}c_r(A)\Big)^2,
\end{split}
\end{equation}
then choose 
$$B\,=\,\frac{8 \|\phi\|_\infty^2M^2}{\varepsilon_0 }\Big(\sup_{0\leq r\leq T}c_r(A)\Big)^2.$$
Note that \eqref{intcW} is just a simple consequence of Lemma \ref{Wrepestimate}.
\end{proof}

\begin{lemma}\label{Wrepestimate}
Fix $T>0$, $a\in\bb R$ and $A\subset\bb Z$. Then, for every $0\leq r\leq T$,   every density $f$ with respect to $\mu_r^n$,  we have
\begin{equation*}
\limsup_{n\to\infty}\Big\{\int  a W_{r,A}^\varepsilon f d\mu_r^n\,-\, \frac{n^2}{4}D(\sqrt{f};\mu_r^n)\Big\}\,\leq\,\frac{4a^2\varepsilon \|\phi\|_\infty^2}{\varepsilon_0 q_r^2}.
\end{equation*}
\end{lemma}
\begin{proof}
Recall the definition of $W_{r,A}^{\varepsilon}$ given in \eqref{defW}. We are going to first estimate the cost of performing the replacement on $\overline{\eta}(A_r)$, then estimate the cost of the replacement on $\overline{\eta}(A_\ell)$.

In view of Lemma \ref{dif}, since $\overline{\eta}(A\backslash\{A_r\})$ is absolutely bounded by $1$ and has a support disjoint with $[A_r,A_r+\varepsilon n]\cap\bb Z$,  we have
\begin{equation*}
\begin{split}
\limsup_{n\to\infty}\Big\{ &a\int\sqrt{n}\Big [\frac{\overline{\eta}(A)}{q_r}\,-\,\frac{\overline{\eta}(A\backslash\{A_r\})}{q_r}  \Big(\frac{1}{n}\sum_{x=1}^{n}\phi_\varepsilon\big(\frac{x}{n}\big) w(x+A_r)\mc X(\rho_r\big(\frac{A_r}{n}\big))\Big)\Big ]f d\mu_r^n\\
-\,&\frac{n^2}{8}D(\sqrt{f};\mu_r^n)\Big\}\,\leq\,\frac{4a^2\varepsilon \|\phi\|_\infty^2}{\varepsilon_0 q_r^2}.
\end{split}
\end{equation*}
To conclude the proof of this lemma, it remains to show 
\begin{equation}\label{repdegtwo}
\begin{split}
\limsup_{n\to\infty}\Big\{ &\frac{a\sqrt{n}}{q_r}\int\Big [\overline{\eta}(A_l)\,-\,\frac{1}{n^{1/6}}\sum_{x=1}^{n^{1/6}} w(-x+A_\ell) \mc X(\rho_r\big(\frac{A_\ell}{n}\big))\Big] \overline{\eta}(A\backslash\{A_l,A_r\})\\
\times \Big(\frac{1}{n}\sum_{x=1}^{n}&\phi_\varepsilon\big(\frac{x}{n}\big) w(x+A_r)\mc X(\rho_r\big(\frac{A_r}{n}\big))\Big)f d\mu_r^n\,-\,\frac{n^2}{8}D(\sqrt{f};\mu_r^n)\Big\}\,\leq\,0.
\end{split}
\end{equation}
Writting $w(A_l)-w(-x+A_l)$ as a telescope sum, we have
\begin{equation*} 
\begin{split}
&\overline{\eta}(A_l)\,-\,\frac{1}{n^{1/6}}\sum_{x=1}^{n^{1/6}} w(-x+A_\ell) \mc X(\rho_r\big(\frac{A_\ell}{n}\big))\\
=\,& \mc X(\rho_r\big(\frac{A_\ell}{n}\big))\sum_{x=1}^{n^{1/6}}\big[w(A_l-x+1)-w(A_l-x)\big]\sum_{y=x}^{n^{1/6}} \frac{n^{1/6}-y+1}{n^{1/6}}.
\end{split}
\end{equation*}
By the integration by parts formula obtained in Lemma E.3 in \cite{jm}, using the fact
$$R(A)\,:=\, \overline{\eta}(A\backslash\{A_l,A_r\})\frac{1}{n}\sum_{x=1}^{n}\phi_\varepsilon\big(\frac{x}{n}\big) w(x+A_r)\,\leq\, \frac{2\|\phi\|_\infty }{\varepsilon_0},$$
  the first term inside the limsup in \eqref{repdegtwo} is bounded by
\begin{equation}\label{bypartsdegtwo}
\begin{split}
&\frac{\mc X(\rho_r(A_l))\mc X(\rho_r(A_r))\sqrt{n}}{q_r}\Big\{\sum_{x=A_l-n^{1/6}}^{ A_l-1} \delta n^2D_{x,x+1}(\sqrt{f};\mu_r^n)\,+\, \frac{4a^2 }{ \delta\varepsilon_0 n^2}\frac{4\delta_x^2\|\phi\|_\infty^2}{\varepsilon_0^2}\\
&\,-\, a\delta_x\Big[\rho_r\big(\frac{x+1}{n}\big)-\rho_r\big(\frac{x}{n}\big)\Big]\int R(A)w(x)w(x+1)fd\mu_r^n\Big\}
\end{split}
\end{equation}
for any $\delta>0$, where 
$$\delta_x\,=\,\sum_{y=A_l-x}^{n^{1/6}} \frac{n^{1/6}-y+1}{n^{1/6}}.$$
Choosing $\delta=n^{-2/3}$, using the bound $\delta_x\,\leq\, n^{1/6}$,  the expression in \eqref{bypartsdegtwo} is bounded by
$$C\Big\{n^{11/6}D(\sqrt{f};\mu_r^n)\,+\,n^{-1/3}\Big\},$$
where $C$ is a constant depending on $a, \rho_r, q_r, \varepsilon_0$ and $\phi$. This proves \eqref{repdegtwo}.
\end{proof}

\section{Tightness }\label{sec4}
\subsection{Tightness of $\{Z_t^{h,\varepsilon}: t\in [0,T]\}_{\varepsilon}$}\label{tightZt}

We start by introducing the concept of subgaussian variables.  We say that a real-valued random variable $X$ is \textit{subgaussian} of order $\sigma^2$, if for every $\theta\in \bb R$,
$$\log E[e^{\theta X}]\,\leq\,\frac{1}{2}\sigma^2\theta^2.$$
As a simple consequence of Hoeffding's Lemma(see Lemma \ref{Hf} in Appendix), if $X$ is a Bernoulli random variable with mean $\rho$, then $X-\rho$ is subgaussian of order $1/4$. From \eqref{subgw}  we can easily deduce that, there exists a constant $C=C(\varepsilon_0,\phi)$ such that 
\begin{equation}\label{subw}
\frac{1}{\sqrt{n}} \phi_\varepsilon\big(\frac{x}{n}\big) w(x) \quad \text{is subgaussian of order} \quad \frac{C}{\varepsilon^2 n},
\end{equation}
for every $0\leq x\leq \varepsilon n$.
 We refer to Appendix F.3 of \cite{jm} for a more detailed discussion on subgaussian variables.

\begin{lemma}\label{Z}
Fix $0\leq s < t \leq T$. Then for every $\lambda\in(1,2)$, there exists a positive constant $C>0$ independent of $n,\varepsilon,t$ and $s$  such that
$$\bb E_{\mu^n}\Big[\big|Z_{t,n}^{h,\varepsilon}\,-\,Z_{s,n}^{h,\varepsilon}\big|^\lambda\Big]\,\leq\, C|t-s|^\lambda\varepsilon^{-\lambda/2}$$
for all $n\in\bb N$.
\end{lemma}
\begin{proof}
For every $0\leq r\leq T$, define
$$H_r\,=\,\frac{\mc X (\rho_r(0))\varphi'_h(\rho_s(0))}{q_r}.$$
The expectation to be estimated in the lemma can be written as 
$$\bb E_{\mu^n}\Big[\Big\lvert\int _s^t \frac{1}{\sqrt{n}}\sum_{x=1}^{n}\phi_\varepsilon\big(\frac{x}{n}\big) w_{rn^2}(x)H_r dr\Big\rvert^\lambda \Big]$$
By Jensen's inequality, it is bounded from above by
\begin{equation*}
(t-s)^{\lambda-1}\int_s^t\bb E_{\mu^n}\Big[\Big\lvert \frac{1}{\sqrt{n}}\sum_{x=1}^{n}\phi_\varepsilon\big(\frac{x}{n}\big) w_{rn^2}(x)H_r\Big\rvert^\lambda \Big]dr.
\end{equation*}
Therefore it is enough to prove that there exists a constant $C>0$ independent of $n$ and $r$ such that
$$\bb E_{\mu^n}\Big[\Big\lvert \frac{1}{\sqrt{n}}\sum_{x=1}^{n}\phi_\varepsilon\big(\frac{x}{n}\big) w_{rn^2}(x)H_r\Big\rvert^\lambda \Big]\,\leq\, C\varepsilon^{-\lambda/2}.$$
In view of Lemma \ref{pe}, since $\lambda>1$, we just need to show 
\begin{equation}\label{prob1}
\bb P_{\mu^n}\Big[\Big|\frac{1}{\sqrt{n}}\sum_{x=1}^{n}\phi_\varepsilon\big(\frac{x}{n}\big) w_{rn^2}(x)H_r\Big|>\Delta\Big]\,\leq\, \frac{C}{\varepsilon \Delta^2}
\end{equation}
for every $\Delta>0$.

Thanks to Proposition \ref{entpro}, the probability at the left hand side of \eqref{prob1} is less than or equal to
\begin{equation}\label{stepent}
\frac{H_n(r)\,+\,\log 2}{\log \Big(1\,+\,\mu_r^n\Big[\Big|\frac{1}{\sqrt{n}}\sum_{x=1}^{n}\phi_\varepsilon\big(\frac{x}{n}\big) w_{rn^2}(x)H_r\Big|>\Delta\Big]^{-1}\Big)}.
\end{equation}
 By \eqref{subw} and Lemma \ref{Hi} , since $\mu_r^n$ is a product measure and $|H_r|$ is uniformly bounded, we have
\begin{equation}\label{steph}
\mu_r^n\Big[\Big|\frac{1}{\sqrt{n}}\sum_{x=1}^{n}\phi_\varepsilon\big(\frac{x}{n}\big) w_{rn^2}(x)H_r\Big|>\Delta\Big]\,\leq\,2\exp\Big( -\frac{\Delta^2}{C\sum_{x=1}^{\varepsilon n} n^{-1}\varepsilon^{-2}}\Big).
\end{equation}
It follows from \eqref{stepent} and \eqref{steph} that, the probability at the left hand side of \eqref{prob1} is bounded by
$$\frac{H_n(r)\,+\,\log 2}{\log\big(1+\frac{1}{2}\exp\{C\varepsilon \Delta^2\}\big)}.$$
Recall that it is proved in Theorem \ref{entropy} that $H_n(r)$ is of order $O(1)$ for every $0\leq r\leq T$.  Using the elementary inequality 
$$\log(1+\frac{e^x}{2})\geq \frac{x}{3}, \quad \forall \,\, x\geq 0,$$
we conclude the proof of \eqref{prob1}.
\end{proof}

\begin{lemma}\label{Zdif}
For every $\lambda\in (1,2)$, there exists a positive constant $C=C(\lambda)$ independent of $\delta,\varepsilon$ and $t$ such that
$$\bb E_{\mu^n}\Big[\big\lvert Z_{t,n}^{h,\delta}\,-\,Z_{t,n}^{h,\varepsilon}\big\rvert^\lambda\Big]\,\leq\,Ct^{\frac{\lambda}{2}}\varepsilon^{\frac{\lambda}{2}}$$
for all $0<\delta<\varepsilon<1$ and all $n\in\bb N$.
\end{lemma}
\begin{proof}
Notice that the expectation in the lemma is bounded by a constant $C(\lambda)$ times the sum
\begin{equation*}
\bb E_{\mu^n}\Big[\big|\widetilde{\Gamma}^h_n(t)\,-\,Z_{t,n}^\varepsilon\big|^\lambda\Big]\,+\,\bb E_{\mu^n}\Big[\big|Z_{t,n}^\delta\,-\,\widetilde{\Gamma}^h_n(t)\big|^\lambda\Big]
\end{equation*}
 Each term can be bounded by  $Ct^{\frac{\lambda}{2}}\varepsilon^{\frac{\lambda}{2}}$ in view of Lemma \ref{ZG}, since $\delta<\varepsilon$ and $\lambda\geq 1$.
\end{proof}

We now show that the sequence $\{Z_t^{h,\varepsilon}: t\in[0,T]\}_{\varepsilon\in(0,1)}$ is tight with respect to the uniform topology in $C([0,T];\bb R)$.
\begin{theorem}\label{Ztbound}
Given $\lambda\in(1,2)$, there exists a constant $C=C(\lambda )$  such that 
\begin{equation}\label{difZt}
\sup_{\varepsilon\in(0,1)}\bb E_{\mu^n}\Big[\big|Z_{t,n}^{h,\varepsilon}\,-\,Z_{s,n}^{h,\varepsilon}\big|^\lambda\Big]\,\leq\,C (t-s)^{\frac{3\lambda}{4}}
\end{equation}
for all $0\leq s<t\leq T$ and all $n\in\bb N$.
\end{theorem}
\begin{proof}
By Lemma \ref{Z}, there exists a constant $C=C(\lambda)>0$ such that 
\begin{equation}\label{Zlambda}
\begin{split}
\bb E_{\mu_n}\Big[\big\lvert Z_t^{h,\varepsilon}\big\rvert^\lambda\Big]\,\leq\,&\,\bb E_{\mu^n}\Big[\big\lvert Z_t^{h,\varepsilon}\,-\,Z_{t,n}^{h,\varepsilon}\big\rvert^\lambda\Big] \,+\,\bb E_{\mu^n}\Big[\big\lvert Z_{t,n}^{h,\varepsilon}\big\rvert^\lambda\Big]\\
\leq\,& Ct^\lambda\varepsilon^{-\lambda/2}.
\end{split}
\end{equation}

Fix $0<\varepsilon<1$. Given $\delta<\varepsilon$, in view of Lemma \ref{Z} and Lemma \ref{Zdif}, 
\begin{equation*}
\begin{split}
\bb E_{\mu^n}\Big[\big\lvert Z_{t,n}^{h,\delta}\big\rvert^\lambda\Big]\,\leq\, &C(\lambda)\bb E_{\mu^n}\Big[\big\lvert Z_{t,n}^{h,\varepsilon}\big\rvert^\lambda\Big]\,+\, C(\lambda)\bb E_{\mu^n}\Big[\big\lvert Z_{t,n}^{h,\varepsilon}-Z_{t,n}^{h,\delta}\big\rvert^\lambda\Big]\\
\,\leq\, &Ct^\lambda\varepsilon^{-\lambda/2}\,+\, Ct^{\frac{\lambda}{2}}\varepsilon^{\frac{\lambda}{2}}.
\end{split}
\end{equation*}
If $t\geq \delta^2$, let $\varepsilon=\sqrt{t/T}$, then $\bb E_{\mu^n}\Big[\big\lvert Z_t^{h,\delta}\big\rvert^\lambda\Big]\,\leq\, C(\lambda,T)t^{\frac{3\lambda}{4}}$. If $t<\delta^2$, using \ref{Zlambda}, we still get the same upper bound of $\bb E_{\mu^n}\Big[\big\lvert Z_t^{h,\delta}\big\rvert^\lambda\Big]$.

Since we have proved a uniform bound on $H_n(t)\leq C$ for all $t\in[0,T]$ and this entropy bound is the only property of the initial measure needed to obtain the bound of $\bb E_{\mu^n}\Big[\big\lvert Z_t^{h,\delta}\big\rvert^\lambda\Big]$, shifting the time, 
\begin{equation*}
\begin{split}
\bb E_{\mu^n}\Big[\big|Z_{t,n}^{h,\varepsilon}\,-\,Z_{s,n}^{h,\varepsilon}\big|^\lambda\Big]\,=\,&\bb E_{\eta_{sn^2}}\Big[\big|Z_{t-s,n}^{h,\varepsilon}\,-\,Z_{0,n}^{h,\varepsilon}\big|^\lambda\Big]\\
\,=\,&\bb E_{\eta_{sn^2}}\Big[\big|Z_{t-s,n}^{h,\varepsilon}\big|^\lambda\Big]\,\leq\,C(t-s)^{\frac{3\lambda}{4}}.
\end{split}
\end{equation*}

\end{proof}
By Theorem \ref{CLT}, letting $n\to\infty$ at the left hand side of \eqref{difZt}, we can bound the $L^\lambda$ norm of the $Z_{t}^{h,\varepsilon}-Z_{s}^{h,\varepsilon}$ by $C (t-s)^{\frac{3\lambda}{4}}$. Choosing $\lambda\in(4/3,2)$, the tightness of $\{Z_t^{h,\varepsilon}: t\in[0,T]\}_{\varepsilon\in(0,1)}$ follows from this estimate and Kolmogorov-Centov criterion(see Problem 2.4.11 in \cite{ks}).

\subsection{Estimates on the error process}\label{ssec42}
In this subsection we will show that, for every fixed $\varepsilon\in(0,1)$, the process $\{J_n^{h,\varepsilon}(t); 0\leq t\leq T\}$ converges in law to a zero process $\{0; 0\leq t\leq T\}$ as $n\to\infty$. In addition, we provide an estimate to the $L^2$ norm of $J_n^{h,\varepsilon}(t)\,-\,J_n^{h,\delta}(s)$, which will be used in the proof of the tightness of $\{\Gamma_n^h: t\in[0,T]\}_n$ in the next subsection.

First of all, from claim \ref{claimdec}, we have
$$\sup_{0\leq t\leq T}|J_{n,1}^{h}(t)| \quad \text{is of order}\,\, O(n^{-1/2}).$$
 This implies that 
\begin{equation}\label{unifJ1}
\bb E_{\mu^n}\Big[\big\lvert J_{n,1}^{h}(t)\,-\,J_{n,1}^{h}(s)\big\rvert^\lambda\Big]\,\lesssim\, (t-s)^\lambda n^{-\lambda} .
\end{equation}
for every $\lambda>0$.
Thus  the sequence of processes $\{J_{n,1}^h(t): t\in[0,T]\}_{n\in\bb N}$ is tight with respect to the uniform topology and the limit is a zero process. In what follows we will bound the $L^2$ norm of $J_{n,i}^{h,\varepsilon}(t)-J_{n,i}^{h,\varepsilon}(t)$, $i=2,3$, by $C(t-s)$ times $n$ to some negative power. Then it follows that $\{J_n^{h,\varepsilon}(t); 0\leq t\leq T\}$ converges in law to a zero process.

We start with the analysis on the process $\{J_{n,3}^{h,\varepsilon}(t): t\in[0,T]\}$.
\begin{lemma}\label{L2J3}
For every $A\subset[-R,R]\cap\bb Z$ such that $|A|\geq 2$, there exists a constant $C$ independent of $n$ and $\varepsilon$ such that 
$$\bb E_{\mu^n}\Big[ \Big\lvert \sqrt{n}\int_s^t  c_r(A)Q_{r,A}^{\varepsilon}dr\Big\rvert^2 \Big]\,\leq\, \frac{C(t-s)^2}{n^{1/12}\varepsilon^{1/2}}.$$
\end{lemma}
\begin{proof}
By Jensen's inequality, the expectation to be estimated in the lemma is less than or equal to
$$(t-s)\int_s^t \bb E_{\mu^n}\Big[  \sqrt{n}\Big\lvert c_r(A)Q_{r,A}^{\varepsilon}\Big\rvert \Big]dr.$$
Since $c_r(A)$ is uniformly bounded over $r\in[0,T]$ and $A\subset[-R,R]\cap\bb Z$, using the fact $|\overline{\eta}(x)|\leq 1$, the expression above is bounded by
$$C(t-s) \int_s^t \bb E_{\mu^n}\Big[ \sqrt{n}\Big\lvert  \widetilde{Q}_{r,A}^\varepsilon\Big\rvert \Big]dr,$$
where 
$$\widetilde{Q}_{r,A}^\varepsilon\,=\,\Big(\frac{1}{n^{1/6}}\sum_{x=1}^{n^{1/6}} w(-x+A_l)\mc X(\rho_r\big(\frac{A_l}{n}\big))\Big) \Big(\frac{1}{n}\sum_{x=1}^{n}\phi_\varepsilon\big(\frac{x}{n}\big) w(x+A_r)\mc X(\rho_r\big(\frac{A_r}{n}\big))\Big).$$
By entropy inequality and the bound $H_n(r)\leq C$, the expectation inside the integral is bounded by
$$\frac{C}{\gamma}\,+\,\frac{1}{\gamma}\log \mu_r^n\Big[\exp\Big\{\gamma \sqrt{n}\Big\lvert\widetilde{Q}_{r,A}^\varepsilon\Big\rvert \Big\}\Big] $$
for any $\gamma>0$. 
Note that by \eqref{subgw}  and Lemma \ref{subgsum} we can easily deduce that under the Bernoulli product measure $\mu_r^n$, there exists a constant $C_1=C_1(\varepsilon_0,M,\phi)$ such that \begin{equation}\label{subwr}
\frac{1}{n}\sum_{x=0}^{\varepsilon n}\phi_\varepsilon\big(\frac{x}{n}\big) w(x)\frac{\mc X(\rho_r(0))}{q_r} \quad \text{is subgaussian of order} \quad \frac{C_1}{\varepsilon n},
\end{equation}
and a constant $C_2=C_2(\varepsilon_0,M,\phi, A)$ such that
\begin{equation}\label{subwl}
\frac{1}{n^{1/6}}\sum_{x=1}^{n^{1/6}} w(-x+A_l)\mc X(\rho_r\big(\frac{A_l}{n}\big))
 \quad \text{is subgaussian of order} \quad \frac{C_2}{n^{1/6}}
\end{equation}
for every $1\leq x\leq \varepsilon n$. 
Using elementary inequalities
$$e^{|x|}\,\leq \,e^x+e^{-x}$$
and
$$\log(a+b)\,\leq\,\log a\,+\,\log b\,+\,\log 2\quad \text{for all}\,\, a,b>0,$$
by \eqref{subwr} \eqref{subwl} and Lemma \ref{subgprod}, choosing 
$$\gamma=\frac{4n^{1/12}}{\sqrt{\varepsilon C_1C_2}},$$
  the previous expression is bounded by $Cn^{-1/12}\varepsilon^{-1/2}$.
Putting all the estimates above together, we conclude that 
$$\bb E_{\mu^n}\Big[ \Big\lvert \sqrt{n}\int_s^t  c_r(A)Q_{r,A}^{\varepsilon}dr\Big\rvert^2 \Big]\,\leq\, \frac{C(t-s)^2}{n^{1/12}\varepsilon^{1/2}}.$$

\end{proof}
Recall that $ J_{n,3}^{h,\varepsilon}(t)$ is defined as the time integral of the (finite) sum of $\sqrt{n}c_r(A)Q_{r,A}^\varepsilon$. By an triangle inequality, from the lemma above we get
\begin{equation}\label{L2J3sum}
\bb E_{\mu^n}\Big[ \Big\lvert J_{n,3}^{h,\varepsilon}(t)\,-\, J_{n,3}^{h,\varepsilon}(s)\Big\rvert^2 \Big]\,\leq\, \frac{C(t-s)^2}{n^{1/12}\varepsilon^{1/2}}.
\end{equation}

The estimate of $J_{n,2}^{h,\varepsilon}$ is similar to the proof of Lemma \ref{L2J3}. Actually it is even simplier because there is no need to deal with the exponential moment of the product of two subgaussian random variables. Here we omit the proof and state the result only.
\begin{lemma}\label{L2J2}
There exists a constant $C$ independent of $n$ and $\varepsilon$ such that
$$\bb E_{\mu^n}\Big[\big\lvert J_{n,2}^{h,\varepsilon}(t)\,-\,J_{n,2}^{h,\varepsilon}(s)\big\rvert^2\Big]\,\lesssim\, (t-s)^2 n^{-1} \varepsilon^{-1/2}.$$
\end{lemma}

If one checks carefully the proof of Theorem \eqref{Ztbound} and related previous  lemmas, the constant $C$ at the right hand side of \eqref{difZt}  depends on the absolute value of $\varphi'_h(\rho_r(0))$. The explicit expression of   $\varphi'_h(\rho_s(0))$ does not really matter. Actually $C^{1/\lambda}$ is proportional to
$$\sup_{s\leq r\leq t}\| \varphi'_h(\rho_r(0)) \|_\infty.$$ From this observation and claim \eqref{claimdec} we can conclude that, given any $\lambda\in(1,2)$, there exists a constant $C=C(\lambda)$ such that 
\begin{equation}\label{unifJ2}
\sup_{\varepsilon\in(0,1)}\bb E_{\mu^n}\Big[\big\lvert J_{n,2}^{h,\varepsilon}(t)\,-\,J_{n,2}^{h,\varepsilon}(s) \big\rvert^\lambda\Big]\,\leq\, C (t-s)^{\frac{3\lambda}{4}}n^{-\lambda},
\end{equation}
for all $n\in\bb N$.

By Lemma \ref{pe} and Lemma \ref{termtwo}, we have
\begin{equation*}
\begin{split}
\bb E_{\mu^n}\Big[\big\lvert J_{n,3}^{h,\varepsilon}(t) \,-\,\int_0^t \sum_{\substack{A\subset[-R,R]\cap\bb Z\\|A|\geq 2}} \sqrt{n} c_r(A)\frac{\overline{\eta}_{rn^2}(A)}{q_r}dr\big\rvert^\lambda\Big] 
\leq\, Ct^{\lambda/2}\varepsilon^{\lambda/2}
\end{split}
\end{equation*}
for every $\lambda\in(1,2)$. Similar to the proof of Lemma \ref{Zdif}, a triangle inequality gives
$$\bb E_{\mu^n}\Big[\big\lvert J_{n,3}^{h,\varepsilon}(t)\,-\,J_{n,3}^{h,\delta}(t)\big\rvert^\lambda\Big]\,\leq\,Ct^{\frac{\lambda}{2}}\varepsilon^{\frac{\lambda}{2}}$$
for all $0<\delta<\varepsilon<1$. With this estimate at hand, one can apply the same argument in the proof of Theorem \ref{Ztbound} to conclude that, given any $\lambda\in(1,2)$, there exists a constant $C=C(\lambda)$ such that 
\begin{equation}\label{unifJ3}
\sup_{\varepsilon\in(0,1)}\bb E_{\mu^n}\Big[\big\lvert J_{n,3}^{h,\varepsilon}(t)\,-\,J_{n,3}^{h,\varepsilon}(s) \big\rvert^\lambda\Big]\,\leq\, C (t-s)^{\frac{3\lambda}{4}},
\end{equation}
for all $n\in\bb N$.

\subsection{Tightness of $\{\Gamma^{h}_n(t):t\in[0,T]\}_n$}\label{ssec43}

The tightness of $\{\Gamma^{h}_n(t):t\in[0,T]\}_n$ with respect to the uniform topology in space $D([0,T]; \bb R)$ is now easy based on our previous estimates. Again We use Kolmogorov-Centov criterion and the following theorem to prove.
\begin{theorem}\label{tightTau}
Fix $0\leq s < t \leq T$. Given any $\lambda\in(1,2)$,  there exists a constant $C = C( \lambda )$ such that
$$\bb E_{\mu^n}\Big[\big|\Gamma^{h}_n(t)-\Gamma^{h}_n(s)\big|^\lambda\Big]\,\leq\, C(t-s)^{\frac{3\lambda}{4}}$$
holds for all $n\in\bb N$. 
\end{theorem}
\begin{proof}
It is enough to prove the theorem only for the case $t-s<1$. Using triangle inequality, the expectation in the theorem is bounded by a constant $C(\lambda)$ times the sum 
\begin{equation*}
\begin{split}
&\bb E_{\mu^n}\Big[\big|Z_{t,n}^{h,\varepsilon}\,-\,Z_{s,n}^{h,\varepsilon}\,-\,\big(\widetilde{\Gamma}^{h,\varepsilon}_n(t)-\widetilde{\Gamma}^{h,\varepsilon}_n(s)\big)\big|^\lambda\Big]\\
+\,&\bb E_{\mu^n}\Big[\big|Z_{t,n}^{h,\varepsilon}\,-\,Z_{s,n}^{h,\varepsilon}\big|^\lambda\Big]\,+\, \bb E_{\mu^n}\Big[\big|J_{n}^{h,\varepsilon}(t)\,-\,J_{n}^{h,\varepsilon}(s)\big|^\lambda\Big] .
\end{split}
\end{equation*}
Choosing $\varepsilon=\sqrt{t-s}$ in Lemma \ref{ZG}, the first expectation is bounded by $C(t-s)^{\frac{3\lambda}{4}}$. The second expectation is also bounded by $C(t-s)^{\frac{3\lambda}{4}}$ by Theorem \ref{Ztbound}. The third expectation has the same upper bound as well because of \eqref{unifJ1},\eqref{unifJ2},\eqref{unifJ3}.
\end{proof}

\section{The limit}\label{sec5}
In this section we are going to prove Theorem \ref{limit} and \ref{time}.We start by showing that $X_t(f)$ is Gaussian assuming $X_0$ is a Gaussian random field.

Recall the definition of $\bb L_t$ given in \eqref{Lt}. Given $f\in C^\infty(\bb T)$ and $t\in[0,T]$, let $\{P_{s,t}f: 0\leq s\leq t\}$ be the solution of the backwards Fokker-Planck equation
\begin{equation*}
\begin{cases}
\partial_s v_s\,+\,\bb L_sv_s\,=\,0 \quad \text{for} \,\, s\leq t\\
v_t\,=\,f
\end{cases}
\end{equation*}
It follows from Theorem 5.1 of Chapter IV in \cite{lsu} that, if $f$ is smooth, then  $P_{s,t}f(\cdot)$ is a smooth function on $\bb T$ and $s\mapsto P_{\cdot, t}f(u)$ is smooth on $[0,t]$ for any $u\in \bb T$.

It is shown in section 6.3 of \cite{jm} that $X_t(f)$ can be represented as the sum of two independent variables:
$$X_t(f)\,=\, X_0(P_{0,t}f)\,+\,M_t(P_{\cdot,t}f)$$
for every $f\in C^\infty(\bb T)$. By the quadratic variation formula given in \eqref{qvar} and L\'evy's characterization theorem(see  Theorem II.4.4 of \cite{js}), $\{M_t(P_{\cdot,t}f): t\geq 0\}$ is a Gaussian process. Therefore, as long as $X_0$ is a Gaussian random field, then $\{X_t(f): t\geq0\}$ is a Gaussian process.

Finally we are in a position to prove Theorem \ref{limit} and \ref{time}.
\begin{proof}[Proof of Theorem \ref{limit}]
We have shown in subsection \ref{tightZt} that the sequence the processes $\{Z_t^{h,\varepsilon}: t\in [0,T]\}_{\varepsilon>0}$ is tight with respect to the uniform topology of $C([0,T]; \bb R)$. Suppose that $\{Z^h_t: t\in [0,T]\}$ is one of the limits. Fix any $t\in[0,T]$. By Lemma \ref{Zdif} and Theorem \ref{CLT}, taking $n\to\infty$, $\{Z_t^{h,\varepsilon}\}_{\varepsilon>0}$ is a Cauchy sequence in $L^\lambda(\bb P)$ for any $\lambda\in(1,2)$. Therefore $Z_t^{h,\varepsilon}$ converges to some limit as $\varepsilon\to0$ and this limit has to be $Z^h_t$. Since the law of a continuous process is determined by its finite dimensional distributions, this proves the uniqueness of the limit $\{Z^h_t: t\in [0,T]\}$ in $C([0,T];\bb R)$.

Since $\{X_t(f), t\geq 0\}$ is a Gaussian process, by definition of $Z_t^{h,\varepsilon}$, $\{Z_t^{h,\varepsilon}: t\in[0,T]\}$ is a Gaussian process. Since $Z_t^{h,\varepsilon}$ converges to $Z_t^h$ in $L^\lambda$ norm with $\lambda\in(1,2)$, $\{Z^h_t:t\in[0,T]\} $ is a Gaussian process as well.
\end{proof}

\begin{proof}[Proof of Theorem \ref{time}]
The tightness of $\{\Gamma^h_n(t):t\in[0,T]\}_{n\in\bb N}$ with respect to the uniform topology in space $D([0,T]; \bb R)$ was proved in subsection \ref{ssec43}. Let $\{\Gamma^h(t):t\in[0,T]\}$ be a limit point of $\{\Gamma^h_n(t):t\in[0,T]\}_n$. Recall that 
$$\Gamma_n^h(t)\,=\,\widetilde{\Gamma}_n^{h,\varepsilon}(t)\,+\, J_n^{h,\varepsilon}(t)$$
for every $\varepsilon\in(0,1)$. Since both $\{\Gamma^h_n(t):t\in[0,T]\}_{n\in\bb N}$ and $\{J^{h,\varepsilon}_n(t):t\in[0,T]\}_{n\in\bb N}$ are tight, $\{\widetilde{\Gamma}^h_n(t):t\in[0,T]\}_{n\in\bb N}$ is also tight. Moreover we have shown in subsection \ref{ssec42} that $J_n^{h,\varepsilon}$ vanishes as $n\to\infty$ for any $\varepsilon\in(0,1)$.  Therefore  $\{\Gamma^h(t):t\in[0,T]\}$ is also a limit point of $\{\widetilde{\Gamma}^{h,\varepsilon}_n(t):t\in[0,T]\}_n$ for any $\varepsilon\in(0,1)$.

 Without loss of generality, let us assume that $\{\Gamma^h(t):t\in[0,T]\}$  is defined in the same probability space on which the process $\{X_t: t\in[0,T]\}$ is defined. For any $\lambda\in (1,2)$, since $L^\lambda$ upper bounds are preserved by convergence in distribution,  by Lemma \ref{ZG} and Theorem \ref{CLT}, taking $n\to\infty$,
$$\bb E\Big[\big(\Gamma^h(t)\,-\,Z_t^{h,\varepsilon}\big)^\lambda\Big]\,\leq\, C t^{\frac{\lambda}{2}}\varepsilon^{\frac{\lambda}{2}}.$$
Sending $\varepsilon\to 0$, it follows that $\{\Gamma^h(t): t\in [0,T]\}$ has the same finite dimensional distributions as those of $\{Z^h_t: t\in [0,T]\}$ and we finish the proof.
\end{proof}

\section{Appendix}

In this appendix we discuss the subgaussian random variables and some properties of them. Most of the results presented here are taken from Appendix F of \cite{jm}, so we omit their proofs.  

\begin{definition}
We say that a real-valued random variable $X$ is subgaussian of order $\sigma^2$, if for every $\theta\in \bb R$,
$$\log E[e^{\theta X}]\,\leq\,\frac{1}{2}\sigma^2\theta^2.$$
\end{definition}

The next lemma, known as the Hoeffding's lemma, provides a class of examples of subgaussian random variables.

\begin{lemma}\label{Hf}
Let $X$ be a random variable taking values in $[0,1]$. Then for any $\theta\in\bb R$,
$$\log E\big[e^{\theta(X-E[X])}\big]\,\leq\, \frac{1}{8}\theta^2.$$
\end{lemma}

Recall that  $\rho_t(u)\in (\varepsilon_0,1-\varepsilon_0)$ for all $t\in[0,T]$ and all $u\in\bb T$. From this fact and the previous lemma, one can easily deduce that  
\begin{equation}\label{subgw}
w(x) \, \text{ is subgaussian of order}\,  (2/\varepsilon_0)^{-2}
\end{equation}

 The following lemma is a particular case of Lemma F.12 in \cite{jm}.
\begin{lemma}\label{subgsum}
Let $\{X_i;i\in \bb T_n\}$ be independent random variables. Assume that for any $i\in\bb T_n$, $X_i$ is subgaussian of order $\sigma_i^2$. Then for any $f:\bb T_n\to\bb R$, 
$$\sum_{i\in\bb T_n}f_iX_i \,\, \text{is subgaussian of order} \quad 2\sum_{i\in\bb T_n}\sigma_i^2 f_i^2. $$
\end{lemma}
We also need an estimate of the exponential moments of products of two subgaussian random variables.
\begin{lemma}[Lemma F.8 in \cite{jm}]\label{subgprod}
Let $X_i$ be the subgaussian variables of order $\sigma_i^2$, $i=1,2$. Then, for any $\gamma\leq (4\sigma_1\sigma_2)^{-1}$,
$$E[e^{\gamma X_1 X_2}]\,\leq\,3.$$
\end{lemma}

Finally let us prove a version of the well known Hoeffding's inequality that is needed in the proof of Lemma \ref{Z}.
\begin{lemma}\label{Hi}
Let $\{X_i; 1\leq i\leq n\}$ be a sequence of independent subgaussian random variables of order $\sigma_i^2$ respectively. Then for any $t>0$, 
$$P\Big(\Big\lvert \sum_{i=1}^n X_i\Big\rvert \geq t \Big)\,\leq\, 2\exp\Big\{-\frac{t^2}{4\sum_{i=1}^n\sigma_i^2}\Big\}.$$
\end{lemma}
\begin{proof}
Since 
$$P\Big(\Big\lvert \sum_{i=1}^n X_i\Big\rvert \geq t \Big)\,\leq\, P\Big( \sum_{i=1}^n X_i \geq t \Big)\,+\,P\Big( \sum_{i=1}^n -X_i\geq t \Big),$$
it is enough to show each probability at the right hand side of the above inequality is bounded by $\exp\Big\{-\frac{t^2}{4\sum_{i=1}^n\sigma_i^2}\Big\}$.We will just give the proof of 
$$P\Big( \sum_{i=1}^n X_i \geq t \Big)\,\leq\, \exp\Big\{-\frac{t^2}{4\sum_{i=1}^n\sigma_i^2}\Big\}.$$
The inequality for $-\sum_{i=1}^n X_i$ can be proved in the same way.

Since $\{X_i; i\geq 0\}$ are independent, by Lemma \ref{subgsum}, $ \sum_{i=1}^n X_i$ is subgaussian of order $2\sum_{i=1}^n\sigma_i^2$. Let $\lambda>0$ be a parameter that will be chosen later. Then
\begin{equation*}
\begin{split}
P\Big( \sum_{i=1}^n X_i \geq t \Big)\,\leq\,& e^{-\lambda t} E\Big[ e^{\lambda \sum_{i=1}^n X_i}\Big]\\
\leq\,& e^{-\lambda t} e^{\lambda^2 \sum_{i=1}^n\sigma_i^2}
\end{split}
\end{equation*}
Optimizing in $\lambda$ and thus choosing $\lambda= \frac{t}{2\sum_{i=1}^n\sigma_i^2}$, we conclude that 
$$P\Big( \sum_{i=1}^n X_i \geq t \Big)\,\leq\, \exp\Big\{-\frac{t^2}{4\sum_{i=1}^n\sigma_i^2}\Big\} .$$
\end{proof}

\smallskip\noindent{\bf Acknowledgments.} L.R.F. was partially supported by CNPq grant 307884/2019-8, and FAPESP grant 2017/10555-0. T. X. would like to thank the financial support of FAPESP grant 2019/02226-2.


\begin{thebibliography}{99}

\bibitem{kl}
C. Kipnis and C. Landim. {\em Scaling limits of interacting
  particle systems}. Grundlehren der Mathematischen Wissenschaften
  [Fundamental Principles of Mathematical Sciences], 320.
  Springer-Verlag, Berlin, 1999.

\bibitem{jm}
M.Jara and O. Menezes. Non-equiliburim fluctuations of interacting particle systems. arXiv:1810.09526

\bibitem{efx}
D. Erhard, T. Franco and T. Xu.  Nonequilibrium joint fluctuations for current and occupation time in the symmetric exclusion process. arXiv:2304.13790

\bibitem{jm1}
M.Jara and O. Menezes. Symmetric exclusion as a random environment: invariance principle. Annals of Probability. 48(6), 3124-3149, 2020

\bibitem{m}
O. Menezes.Non-equailibrium fluctuations of interacting particle systems. PhD thesis, 2017

\bibitem{gj13}
P. Gon\c calves and M. Jara. Scaling limits of additive functionals of interacting particle systems. Comm. Pure  Appl. Math.  66(5), 649-677, 2013

\bibitem{fgn14}
T. Franco, P. Gon\c calves, and A. Neumann. Occupation time of exclusion processes with conductances. Journal of Statistical Physics, 156 (5), 975-997, 2014

\bibitem{ks}
Karatzas, I. and Shreve, S. Brownian motion and stochastic calculus. Graduate
Texts in Mathematics. 113. Springer-Verlag, New York, second edition,1991

\bibitem{efgnt}
D. Erhard, T. Franco, P. Gon\c calves, A. Neumann, M. Tavares, Non-equilibrium fluctuations for the SSEP with a slow bond.  Ann. Inst. H. Poincaré Probab. Statist. 56(2): 1099-1128, 2020

\bibitem{ry}
D. Revuz and M. Yor. Continuous martingales and Brownian motion, volume 293 of
Grundlehren der Mathematischen Wissenschaften [Fundamental Principles of Mathematical
Sciences]. Springer-Verlag, Berlin, third edition, 1999.

\bibitem{lsu}
O. A. Ladyzenskaja, V. A. Solonnikov, and N. N. Uralceva. Linear and quasilinear equations of
parabolic type. Translated from the Russian by S. Smith. Translations of Mathematical Monographs, Vol. 23. American Mathematical Society, Providence, R.I., 1968.

\bibitem{kv}
C, Kipnis, S.R.S. Varadhan. Central limit theorem for additive functionals of reversible
Markov processes and applications to simple exclusions. Comm. Math. Phys. 104(1),
1-19, 1986

\bibitem{b04}
C. Bernardin. Fluctuations in the occupation time of a site in the asymmetric simple exclusion
process. Ann. Probab. 32(1), 855-879, 2004.

\bibitem{s00}
S. Sethuraman. Central limit theorems for additive functionals of the simple exclusion process.
Ann. Probab. 28(1), 277-302, 2000

\bibitem{s03}
S. Sethuraman. An equivalence of $H_{-1}$ norms for the simple exclusion process. Ann. Probab.
31(1), 35-62,2003

\bibitem{qjs02}
J.Quastel, H. Jankowski, J. Sheriff. Central limit theorem for zero-range processes. Methods
Appl. Anal. 9(3), 393-406, 2002

\bibitem{y91}
H.T. Yau. Relative entropy and hydrodynamics of Ginzburg-Landau models. Lett.
Math. Phys., 22(1),63-80, 1991.

\bibitem{js}
J. Jacod and A. Shiryaev. Limit theorems for stochastic processes, Volume 288. Springer
Science and Business Media, 2013

\end{thebibliography}
\end{document}